\input epsf

\documentclass[10pt]{amsart}

\usepackage{amscd, amsmath, amsthm, enumerate, tabls}
\usepackage[all]{xy}

\def\mylabel#1{\label{#1}   \rlap{\hskip1cm\leftline{#1}}   }
\def\mylabel#1{\label{#1}   \proplabeL{#1} \hskip-3pt  }

\def\mylabel#1{\label{#1}}

\newtheorem{theorem}{Theorem}[section]
\newtheorem{lemma}[theorem]{Lemma}
\newtheorem{proposition}[theorem]{Proposition}

\theoremstyle{definition}     

\newtheorem{claim}[theorem]{Claim}

\newtheorem{remark}[theorem]{Remark}

\numberwithin{equation}{section}

\newcommand{\vP}{{\mathaccent20 {\mP}}}
\newcommand{\vR}{{\mathaccent20 {\mR}}}
\newcommand{\vC}{{\mathaccent20 {\mC}}}
\newcommand{\ve}{{\mathaccent20 {\bf e}}}
\newcommand{\vc}{{\mathaccent20 c}}

\newcommand{\sC}{\mathcal{C}}

\newcommand{\sE}{\mathcal{E}}

\newcommand{\sH}{\mathcal{H}}

\newcommand{\sL}{\mathcal{L}}

\newcommand{\sM}{\mathcal{M}}
\newcommand{\sO}{\mathcal{O}}

\newcommand{\sS}{\mathcal{S}}

\newcommand{\sU}{\mathcal{U}}

\newcommand{\mC}{\mathbb{C}}

\newcommand{\mP}{\mathbb{P}}
\newcommand{\mQ}{\mathbb{Q}}
\newcommand{\mR}{\mathbb{R}}
\newcommand{\mZ}{\mathbb{Z}}

\newcommand{\Aut}{\mathrm{Aut}\,}

\newcommand{\Hom}{\mathrm{Hom}}

\newcommand{\Pic}{\mathrm{Pic}\,}

\newcommand{\sSpec}{\mathcal{S}{\!}pec\,}

\newcommand{\codim}{\mathrm{codim}\,}

\newcommand{\Sing}{\mathrm{Sing}\,}

\begin{document}


\title[
Mirror symmetry and projective geometry of Reye congruence-s I
]{
{
Mirror symmetry and projective geometry of Reye congruences I
}
}

\author[S. Hosono]{Shinobu Hosono }

\address{
Graduate School of Mathematical Sciences, 
University of Tokyo, Komaba Meguro-ku, 
Tokyo 153-8914, Japan
}
\email{ hosono@ms.u-tokyo.ac.jp}

\author[H. Takagi]{ Hiromichi Takagi}

\address{
Graduate School of Mathematical Sciences, 
University of Tokyo, Komaba Meguro-ku, 
Tokyo 153-8914, Japan
}
\email{ takagi@ms.u-tokyo.ac.jp}


\begin{abstract}
Studying the mirror symmetry of a Calabi-Yau threefold $X$ of the 
Reye congruence in $\mP^4$, we conjecture that $X$ has a non-trivial 
Fourier-Mukai partner $Y$. We construct $Y$ as the double cover of a 
determinantal quintic in $\mP^4$ branched over a curve. We also 
calculate BPS numbers of both $X$ and $Y$ (and also a related Calabi-Yau 
complete intersection $\tilde X_0$) using mirror symmetry. 
\end{abstract}

\maketitle


\section{{\bf Introduction }}

The set of lines in the $n$-dimensional projective space parametrized by 
a variety in the Grassmannian $G(2,n+1)$ is called a line congruence.
Classical Reye congruences are the line congruences 
defined by three dimensional linear systems of quadrics in $\mP^3$, 
and have a long history in their study in projective geometry and geometry 
of quadrics, in particular, in relation to Enriques surfaces \cite{Co,CoD}
\footnote{Recently, the derived category of the classical Reye congruence 
has been studied in \cite{IKu, Ku3}.}. 
As the next 
generalization, the Reye congruences defined by five quadrics in $\mP^4$ 
\cite{Ol} 
are of considerable interest since the relevant geometry is given by 
Calabi-Yau threefolds, where another aspect of mirror 
symmetry comes into play in addition to the classical ones. In this paper, 
we study the mirror symmetry of 
the Reye congruences and find that 
every Calabi-Yau threefold of the Reye congruence is paired with another 
Calabi-Yau threefold which arises naturally in the relevant projective 
geometries of the Reye congruence.

In Section 2, after a brief summary of the Reye 
congruences in $\mP^4$, we define our Calabi-Yau threefold $X$ of 
the Reye congruences (we shall call Reye congruence $X$ in short 
hereafter) as a suitable $\mZ_2$ quotient of a generic complete 
intersection $\tilde X_0$ of five symmetric $(1,1)$ divisors in 
$\mP^4 \times \mP^4$. Applying 
the toric method due to Batyrev and Borisov \cite{BB}, we first construct a 
mirror family ${\mathcal X}_0^\vee$ over $\mP^2$ and then reduce this family 
to a diagonal one ${\mathcal X}^\vee$ over $\mP^1$ to obtain the mirror 
of the Reye congruence. We study the period integrals of the both families 
in details, and find several boundary points called large complex 
structure limits \cite{Mo}. 
Among them, we will focus on the 
two distinct boundary points of maximal unipotent monodromy \cite{Sch} 
which appear on the reduced family 
${\mathcal X}^\vee$ over $\mP^1$. 
Using the mirror symmetry, we calculate the Gromov-Witten 
invariants at each boundary point, and also the monodromy of 
the period integrals. Based on these calculations, we identify one of 
the two boundary points with the mirror of the Reye congruence 
$X$ as expected, and for the other boundary point we predict  
a new Calabi-Yau threefold $Y$ that is naturally paired 
with $X$ (Conjecture 1). The existence of $Y$ is quite similar to 
the case of a Calabi-Yau threefold given by a transversal linear section  
of the Grassmannian $G(2,7)$, where $Y$ appears in the Pfaffian variety 
$Pf(7)$, the projective dual of $G(2,7)$ \cite{Ro}. In the case of 
$G(2,7)$ and $Pf(7)$, it has been proved that the two Calabi-Yau threefolds 
have equivalent derived categories of coherent sheaves; 
$D(Coh(X)) \cong D(Coh(Y))$ \cite{BCa,Ku1}. 
We expect that the corresponding property holds also in our 
case (Conjecture 2).


In Section 3, we shall prove our Conjecture 1 constructing 
the predicted Calabi-Yau 
threefold $Y$ in our setting of the Reye congruences. In contrast to 
that our Reye congruence Calabi-Yau threefold $X$ may be defined  
by a $\mZ_2$-quotient of a suitable complete intersection  
of five symmetric $(1,1)$ divisors in $\mP^4 \times \mP^4$, we find that the 
geometry of $Y$ arises naturally as the $\mZ_2$ {\it covering} of a 
determinantal quintic in $\mP^4$ 
ramified along a smooth curve of genus 26 and degree 20. In our 
construction, the projective geometries of the Reye congruences are 
used efficiently, and there it should be clear that our case is 
precisely in the line of the case of $G(2,7)$ and $Pf(7)$ 
studied previously in \cite{BCa,Ku1}. 

In the final section, we will count some curves in $X$ and $Y$ 
which are specific in their geometries, and verify the so-called BPS 
numbers which are read from Gromov-Witten invariants of $X$ and $Y$, 
respectively. 

Main results of this paper are Theorem 3.14 and the BPS numbers of the 
related Calabi-Yau threefolds. 
The BPS numbers of $X$ and $Y$ in 
Tables 1 and 2 in Appendix A, respectively (and also those of $\tilde X_0$ 
in Tables 3,4,5 in Appendix B), are obtained by using mirror symmetry, 
but should be useful as a testing ground for the recent developments in 
the study of integral invariants of Calabi-Yau threefolds (see \cite{PT} 
and references therein).

\vskip0.5cm
\noindent
{\bf Acknowledgments:}
This paper is supported in part by Grant-in Aid Scientific Research 
(C 18540014, S.H.) and Grant-in Aid for Young Scientists (B 20740005, H.T.).

\vskip1.5cm

\section{{\bf Mirror symmetry of Reye congruences} }

\vskip0.5cm
\noindent
{\bf (2-1) } {\it Reye congruences}: 
Let $P=| Q_1,Q_2,Q_3,Q_4,Q_5 |$ be a 4-dimensional linear system of 
quadrics in 
$\mP^4$. For a line $l \subset \mP^4$, we denote by $W_l(P)$ the linear 
system of quadrics in $P$ which contain $l$. 
$P$ is called {\it regular} if the following conditions are 
satisfied:

\begin{enumerate}[(i)]
\item $P$ is basepoint free. 
\item If $l$ is a line in $\mP^4$ which is the vertex of 
some $Q \in P$, then $\dim \,W_l(P)=1$.
\end{enumerate}
Similar to the classical Reye congruence \cite{Co}, 
for a regular system $P$ of quadrics, we define a Reye congruence 
$X$ by 
\begin{equation}
X=\{ \, l \subset \mP^4 \,|\, \dim\, W_l(P)=2 \,\},
\mylabel{eqn:def-Reye}
\end{equation}
(cf. a generalized Reye congruence studied in \cite{Ol}). 
Let $(z,w)$ be the (bi-)homogeneous coordinates of $\mP^4 
\times \mP^4$. The quadrics $Q_i$ define the corresponding bilinear forms: 
\begin{equation}
Q_i(z,w)=\,^t z A_i w , \;\; (i=1,2,..,5) .
\end{equation}
We sometimes identify the quadrics $Q_i$ with their associated $5\times 5$ 
symmetric matrices $A_i$. The bilinear forms define a complete 
intersection in $\mP^4 \times \mP^4$: 
\begin{equation}
\tilde X: Q_1(z,w)=Q_2(z,w)=\cdots=Q_5(z,w)=0 \;\;.
\mylabel{eqn:defCICY}
\end{equation}
Let $\sigma$ be the involution $\sigma: (z,w) 
\leftrightarrow (w,z)$. 
Through the residue formula, the complete intersection $\tilde X$ has  
a holomorphic three form which is invariant under this involution.

\begin{proposition} (cf. \cite[Proposition 1.1]{Ol})
The Reye congruence $X$ is isomorphic to 
$\tilde X/\langle \sigma \rangle$. $\tilde X/\langle \sigma \rangle$ 
is a Calabi-Yau manifold with the Hodge numbers $h^{1,1}(X)=1, h^{2,1}(X)=26$. 
\end{proposition}

\begin{proof}
We first show that $X\simeq \tilde{X}/\langle\sigma\rangle$. 
If $(z,w)\in \tilde{X}$, then, by (i), $z\not =w$
and hence $z$ and $w$ span a line $\langle z, w\rangle$.
Then the linear system of quadrics in $P$ containing $\langle z, w\rangle$
is $\{Q\in P\mid \,^t z Qz=0, \,^t w Qw=0\}$, which has codimension two
by (i). Therefore $\langle z, w\rangle\in X$.
Conversely, let $l$ be a line in $X$.
Then, by the condition (\ref{eqn:def-Reye}), $P$ induces on $l$ a pencil of 
$0$-dimensional quadrics.
Correspondingly, $P$ induces on $l\times l$ a pencil of
$(1,1)$-divisors. Thus there are two base points $(z,w)$ and $(w,z)$ of this
pencil, which generate the line $l=\langle z, w\rangle$,  
and then we have $(z,w)\in \tilde{X}$. 
Thus $X\simeq \tilde{X}/\langle\sigma\rangle$.

Second we show that,
under the condition (i), the condition (ii) implies 
that $\tilde{X}$ is $3$-dimensional and smooth.
Indeed, by the Jacobian criterion,
$\tilde{X}$ is not $3$-dimensional or singular at $(z,w)\in \tilde{X}$
if and only if the line $\langle z, w\rangle$ is contained in the 
vertex of a quadric in $P$. Since $X\simeq \tilde{X}/\langle\sigma\rangle$,
the linear system of quadrics
in $P$ containing $\langle z, w\rangle$ is two-dimensional.
Thus
if the line $\langle z, w\rangle$ is contained 
in the vertex of a quadric in $P$, this is  
a contradiction to (ii). (In fact, under (i), the condition (ii) 
is equivalent to that $\tilde{X}$ is $3$-dimensional and smooth.)

Finally, $h^{1,1}(X)=1$ follows immediately from $h^{1,1}(\tilde X)=2$.  
$h^{2,1}(X)=26$ follows from calculating 
the Euler number $e(\tilde X)=2(h^{1,1}(\tilde X)-h^{2,1}(\tilde X))=-100$ 
for the smooth complete intersection $\tilde X$, and $e(\tilde X)=2 e(X)$, 
$h^{i,0}(X)=h^{i,0}(\tilde X)=0 \,(i=1,2)$.
\end{proof}

By using the isomorphism to the quotient of the 
complete intersection $\tilde X$, we obtain 
the following topological invariants of $X$:  
\[
\deg(X)=35 \,, \; c_2.H=50 \,,\; e(X)=-50\,,
\] 
where $H$ is the ample generator of $Pic(X) \cong \mZ$, and $c_2$ is 
the second Chern class of $X$.

\vskip0.5cm
\noindent
{\bf (2-2)}  {\it Mirror family of the Reye congruence}:
The Reye congruence $X$ has a natural covering 
$\tilde X \xrightarrow{2:1} X$ with $\tilde X$ a Calabi-Yau complete 
intersection (\ref{eqn:defCICY}). We may have a mirror 
(family) to our Reye congruence $X$ 
through the Batyrev-Borisov toric mirror construction \cite{BB} applied to 
$\tilde X$. 

$\tilde X$ is defined by the zero locus 
of the five bilinear forms $Q_i(z,w)$ in $\mP^4 \times \mP^4$. In the 
toric construction, one considers a complete intersection  
of generic polynomial equations of bidegree $(1,1)$ in $\mP^4 \times 
\mP^4$ without imposing the invariance under $\sigma$.  
We denote this generic complete intersection by $\tilde X_0$.  

Consider $\mR^4$ and its dual $\vR^4$ with 
the pairing $\langle \,,\, \rangle: 
\vR^4\times \mR^4 \rightarrow \mR$, and fix the dual bases satisfying 
$\langle \ve_i, {\bf e}_j \rangle = \delta_{ij}$. 
We denote by $\Delta$ the polytope in $\vR^4$ whose integral points 
represent the bases of $H^0(\mP^4, {\mathcal O}(-K_{\mP^4}))$ 
and defines the toric variety $\mP_{\Delta}=\mP^4$. 
The dual polytope  
$\Delta^*:=\{ w \in \mR^4 | \langle v,w \rangle \geq -1 (
v \in \Delta) \}$ may be written simply by 
$Conv.(\{ {\bf e}_1,\cdots, 
{\bf e}_4, -{\bf e}_1-\cdots-{\bf e}_4 \})$.
For the ambient space $\mP^4 \times \mP^4$, we consider the product 
$\Delta \times \Delta$ in $\vR^4 \times \vR^4$.  Then the 
vertices of the dual polytope $(\Delta \times \Delta)^*$ are given by  
$({\bf e}_1, 0),\cdots, ({\bf e}_4, 0), (-{\bf e}_1-\cdots-{\bf e}_4, 0), 
(0,{\bf e}_1),\cdots, (0,{\bf e}_4), (0,-{\bf e}_1-\cdots-{\bf e}_4)$.  
We denote these vertices by $\nu_1,\nu_2, \cdots, \nu_{10}$ in order. 
One may identify these 
integral points with the torus invariant divisors $D_{\nu_i}=:D_i$ 
on $\mP_{\Delta} \times \mP_{\Delta}$.

The integral polytope $\Delta \times \Delta$ is an example of the 
so-called reflexive 
polytope, and the generic complete intersection $\tilde X_{0}$ in 
$\mP_{\Delta \times \Delta}=\mP_\Delta \times \mP_\Delta$ may be  
defined by the data $N_P=\{ \{D_i,D_{i+5} \} \}_{i=1,\cdots,5}$, 
called a {\it nef} partition of the anti-canonical 
class $-K_{\mP_{\Delta\times\Delta}}=D_1+D_2+\cdots+D_{10}$ (or the vertices of 
the dual polytope $(\Delta\times \Delta)^*$). Using the data of the 
nef partition, we have
\begin{equation}
\tilde X_0 : f_{\nabla_1}=f_{\nabla_2}=f_{\nabla_3}=f_{\nabla_4}=f_{\nabla_5}=0 \;,
\end{equation}
where $f_{\nabla_i}$ is a generic element in 
$H^0(\mP_{\Delta \times \Delta}, 
{\mathcal O}(D_i+D_{i+5}))$  and $\nabla_i$ is the sub-polytope in 
$\Delta \times \Delta$ whose 
integral points represent the bases of the cohomology for 
$i=1,\cdots,5$. By the Lefshetz hyperplane theorem and evaluating the 
Euler number of $\tilde X_0$, one see easily that $\tilde X_0$ is a three 
dimensional Calabi-Yau complete intersection with Hodge numbers 
$h^{1,1}(\tilde X_0)=2, h^{2,1}(\tilde X_0)=52$. 

Define $\nabla=Conv.(
\nabla_1,\cdots,\nabla_5)$ in $\vR^4$ and consider its dual 
polytope $\nabla^*$ in $\mR^4$. 
For the nef partition $N_P$, the polytope $\nabla$ is reflexive 
and a choice of its triangulation defines a maximally, projective, 
crepant partial resolution of the associated toric variety 
$\mP_{\nabla^*}$ \cite{BB}. For notational simplicity, we use the same 
notation $\mP_{\nabla^*}$ for such a crepant resolution. 

Since the vertices of $\nabla$ represents the toric invariant divisors 
of the toric variety $\mP_{\nabla^*}$, 
the vertices of $\nabla_i \setminus \{ {\bf 0} \}$ 
define a nef partition $N_P^\vee$ of 
$-K_{\mP_{\nabla^*}}$ and the corresponding line bundles on $\mP_{\nabla^*}$. 
Then the nef partition $N_P^\vee$ determines 
sub-polytopes $\Delta^*_i$ $(i=1,\cdots,5)$ 
in $(\Delta \times \Delta)^*$, representing the global sections of the 
line bundles. By duality, it turns out that $\Delta^*_i 
=Conv(\{ 0, \nu_i, \nu_{i+5} \})$ and we denote by $f_{\Delta_i^*}$ 
the generic global sections of the line bundle over $\mP_{\nabla^*}$. 
The data $\Delta_i^*$ provides us the mirror family of Calabi-Yau complete 
intersections $\tilde X_0$:

\begin{proposition} \mylabel{prop:BB}
The generic complete intersection $\tilde X_0^\vee$ defined by 
\begin{equation}
f_{\Delta_1^*}=f_{\Delta_2^*}=f_{\Delta_3^*}=f_{\Delta_4^*}=f_{\Delta_5^*}=0 
\mylabel{eqn:cicy1}
\end{equation}
in the projective toric variety $\mP_{\nabla^*}$ is a smooth Calabi-Yau 
complete intersection with Hodge numbers $h^{1,1}(\tilde X_0^\vee)=52, 
h^{2,1}(\tilde X_0^\vee)=2$, and defines  
a (topological) mirror Calabi-Yau manifold to $\tilde X_0$. 
\end{proposition}

The toric variety $\mP_{\nabla^*}$ contains the algebraic torus 
$({\mC}^*)^8$ as an open dense subset. We denote the coordinates 
of the tori $({\mC}^*)^4 \times ({\mC}^*)^4$ by $U_i, V_i$ $(i=1, 
\cdots, 4)$, respectively. Then the generic global sections  
(\ref{eqn:cicy1}) may be written by 
\begin{equation}
f_{\Delta_i^*}=a_i+ b_i U_i + c_i V_i \;\;\; 
(i=1,\cdots,5), 
\mylabel{eqn:cicy-def-abc}
\end{equation}
where $U_5, V_5$ are determined by $U_1U_2U_3U_4U_5=1, V_1V_2V_3V_4V_5=1$, 
respectively, and $a_i, b_i, c_i$ are generic parameters. With these 
parameters up to some identifications due to isomorphisms, 
we have a mirror family 
$\tilde {\mathcal X}^\vee_0$ to $\tilde X_0$ after taking the closure 
of the zero locus in $\mP_{\nabla^*}$.  
In particular, when $b_i=c_i$, 
we have a diagonal reduction $\tilde{\mathcal X}^\vee$ of the generic 
family $\tilde{\mathcal X}_0^\vee$. We note that, for this reduced family, 
the defining equations 
are invariant under the involution $\sigma^\vee: U_i \leftrightarrow V_i$ 
of the tori $({\mC}^*)^4 \times ({\mC}^*)^4$. 

\begin{lemma}
There exists a resolution of the ambient toric variety 
$\mP_{\nabla^*}$ where  
the involution $\sigma^\vee$ extends from the tori to $\mP_{\nabla^*}$.
\end{lemma}

\begin{proof}
The polytope $\nabla_i$ in $\nabla=Conv(\nabla_1,\cdots,\nabla_5)$ is 
defined by the property that its integral points represent the bases of 
$H^0(\mP_{\Delta\times\Delta}, {\mathcal O}(D_i+D_{i+5}))$. 
Consider the cohomology $H^0(\mP_\Delta,{\mathcal O}(D_i))$ 
$(1\leq i \leq 5)$. 
Through the support function corresponding to the divisor $D_i$
 \cite{Oda}, it is straightforward to arrive at the polytopes 
\[
\nabla_5^{(1)}=Conv(\{{\bf 0},\ve_1,\ve_2,\ve_3,\ve_4\}) \;,\;\; 
\nabla_i^{(1)}=\nabla_5^{(1)} -\ve_i \; (i=1,\cdots,4)  
\]
to represent the cohomology. Since $\mP_{\Delta\times\Delta}=
\mP_\Delta \times \mP_\Delta$ and 
$H^0(\mP_{\Delta \times \Delta}, {\mathcal O}(D_i+D_{i+5}))$ is  
given by the product, $\nabla_i$ is also given by the product 
$\nabla_i^{(1)} \times \nabla_i^{(1)}$ in $\vR^4 \times \vR^4$. 

Now recall that for the projective toric variety we have 
$\mP_{\nabla^*} = \mP_{\Sigma(\nabla)}$, where $\Sigma(\nabla)$ 
is a fan over the faces of the polytope $\nabla$. Precisely 
a choice of the maximally, projective, crepant partial resolution 
of $\mP_{\nabla^*}$ corresponds to a subdivision of the 
fan $\Sigma(\nabla)$. We translate the involution $\sigma^\vee$ defined on 
$({\mC}^*)^4 \times ({\mC}^*)^4$ as the exchange of the first and 
the second factor of $\vR^4 \times \vR^4$. Our claim follows 
if there exists a subdivision of the fan which is invariant under this 
involution. 

It is clear that the involution acts on the set of vertices of $\nabla$ 
due to the product form $\nabla_i=\nabla_i^{(1)} \times \nabla_i^{(1)}$ 
and the definition $\nabla=Conv(\nabla_1,\cdots,\nabla_5)$. 
The set of faces of the polytope 
$\nabla$ may be determined by a computer program PORTA \cite{PO}, 
for example. 
Then one can verify that the involution in fact acts on the set of faces. 
The last property guarantees that there exists the desired 
subdivision of $\Sigma(\nabla)$. 
\end{proof}

\begin{proposition}
For generic parameters $a_i, b_i=c_i$, the involution 
on $\tilde X^\vee$ is fixed-point free, and the quotient 
$\tilde X^\vee/\langle \sigma^\vee \rangle$ defines a family 
${\mathcal X}^\vee$ of a smooth Calabi-Yau manifolds 
with the Hodge numbers $h^{1,1}=26$, $h^{2,1}=1$. We may regard (and 
will justify) this family as the mirror family of the Reye congruence $X$. 
\end{proposition}

\begin{proof} 
Consider a generic member $\tilde X^\vee$ from the diagonal family 
$\tilde{\mathcal X}^\vee$.  
On the torus $(\mC^*)^8$, using the isomorphisms, one may assume 
its defining equations are given by 
\begin{equation}
f_{\Delta_i^*}=a+U_i+V_i\;\; (i=1,\cdots,5),\;\;  
U_1U_2U_3U_4U_5=1, V_1V_2V_3V_4V_5=1 \;.
\mylabel{eqn:defeq-fa}
\end{equation}
From this, 
we see that there is no fixed point in the torus unless $a=-2$. 
By explicit calculations, one can also verify the same property 
of $\tilde X^\vee$ on each affine chart. 
The Hodge numbers follow from Proposition \ref{prop:BB}. 
\end{proof}

\begin{remark} The Reye congruence $X$ and its mirror $X^\vee$ 
are defined as the quotients by the respective involutions. From these 
constructions, we see that both $X$ and $X^\vee$  have non-trivial 
fundamental groups (of finite orders). 
\end{remark}

\vskip0.5cm

\noindent
{\bf (2-3)} {\it Picard-Fuchs equations 
(``maximally resonant'' GKZ systems)}: 
Let us denote by $\Omega(a,b,c)$ a holomorphic three form for the 
mirror family $\tilde {\mathcal X}_0^\vee$ defined by 
(\ref{eqn:cicy-def-abc}), and consider the period integrals over 
the three cycles. Since the family is that of complete intersections 
in a toric variety, the period integral of a torus cycle can be written 
explicitly \cite{BCo} as 
\begin{equation}
\omega(a,b,c) = \frac{1}{(2\pi \sqrt{-1})^8}
\int_\gamma \frac{a_1a_2a_3a_4a_5}{
f_{\Delta_1^*} 
f_{\Delta_2^*} 
f_{\Delta_3^*} 
f_{\Delta_4^*} 
f_{\Delta_5^*} } \prod_{i=1}^4 \frac{d U_i}{U_i} \frac{dV_i}{V_i}\;,
\end{equation}
where $\gamma$ is the torus cycle $|U_i|=|V_i|=1 (i=1,\cdots,4)$. 
The Picard-Fuchs differential equations satisfied by period integrals  
of complete intersections in toric varieties 
have been studied in general in \cite{HKTY}. 
In particular, we find our present case $\tilde X_0^\vee$ in 
the examples there (Example 4 in Sect.5). 

\begin{proposition} \mylabel{prop:GKZxy}
\begin{enumerate}[(1)]
\item 
The period integral $\omega(a,b,c)$ is a function 
of $x=-\frac{b_1b_2b_3b_4b_5}{a_1a_2a_3a_4a_5}$, 
$y=-\frac{c_1c_2c_3c_4c_5}{a_1a_2a_3a_4a_5}$ and satisfies 
a GKZ hypergeometric system which is given locally by 
\begin{equation*}
\{\theta_x^5-x(\theta_x+\theta_y+1)^5 \} \omega(x,y)=
\{\theta_y^5-y(\theta_x+\theta_y+1)^5 \} \omega(x,y)=0 ,
\end{equation*}
where $\theta_x=x\frac{\partial \;}{\partial x}, 
\theta_y=y\frac{\partial \;}{\partial y}$.  
Globally, this system is defined over $\mP^2$ with $[-x,-y,1]=[u,v,w] 
\in \mP^2$ and the 
monodromy is unipotent about the toric divisors $x=0$ and $y=0$. 
Furthermore the system is reducible and its irreducible part 
determines the all period integrals of the family. 
\item
In the affine local coordinates $[1,-y_1,-x_1]:=[u,v,w] \in \mP^2$ 
and $[-y_2,1,-x_2]:=[u,v,w]$, respectively, the system 
is represented by the differential equations 
\begin{equation*}
\{(\theta_{x_1}-1)^5-x_1(\theta_{x_1}+\theta_{y_1})^5 \} 
\tilde \omega(x_1,y_1)=
\{\theta_{y_1}^5-y_1(\theta_{x_1}+\theta_{y_1})^5 \} 
\tilde \omega(x_1,y_1)=0 ,
\end{equation*}
and the same form of the differential operators for 
$\tilde w(x_2,y_2)$. 
\item
The monodromy is unipotent also about the toric divisor 
$w=0$ $([u,v,w] \in \mP^2)$, and there is only one regular 
solution at every toric boundary point of $\mP^2$. 
\end{enumerate}
\end{proposition}

\begin{proof} For the details of the GKZ system, 
and the derivations of the properties (1), we refer to \cite{HKTY}. 
The compactification of the parameter space follows from the construction 
of the secondary fan \cite{GKZ1}, which turns out to be isomorphic to 
the fan of the toric variety $\mP^2$. The form of the differential operators 
in (2) follows in the same way as (1).  Then the properties in (3) follow  
directly from (2). 
\end{proof}

As summarized in the above proposition, the Picard-Fuchs equation 
of the mirror family $\tilde{\mathcal X}_0^\vee$ is defined over 
$\mP^2$ after a natural compactification of the parameter 
space $a_i,b_i,c_i$. In what follows, we consider our mirror 
family $\tilde{\mathcal X}_0^\vee$ over $\mP^2$. 

In general, mirror families have boundary divisors about which the 
monodromies are of finite order. It should be emphasized that 
our mirror family $\tilde{\mathcal X}_0^\vee$ has unipotent 
monodromies about all the toric boundary divisors. 
In fact, one can check that all the toric boundary points given by 
normal crossings of these, are the so-called {\it Large Complex Structure 
Limit (LCSL)} points (see \cite{Mo} for a precise definition).  
Each of the LCSL point will be interpreted from the geometry of 
$\tilde X_0$ in Sect.(3-1).

\begin{remark} 1) The discriminant locus of the GKZ system 
(1) in Proposition \ref{prop:GKZxy} consists of the boundary toric divisor
$\mP^1$'s  and the genus 6 nodal curve with 6 nodes in $\mP^2$ given by
\[
dis_0:=(u+v+w)^5 - 5^4 \,u\,v\,w\,(u+v+w)^2 + 5^5\,u\,v \,w\, 
( u v + v w + w u) =0.
\]
Two of the 6 nodes lie on the line $u=v$, and similarly on the lines 
$v=w$ and $u=w$ for the rest. 
The two nodes on the line $u=v$ correspond to the conifolds 
at $x=\alpha_1, \alpha_2$ of the diagonal family ${\mathcal X}^\vee$ 
(see (\ref{eqn:P-indices}) below). 

2) The compactification of the parameter space $a_i,b_i,c_i$ 
is standard due to \cite{GKZ1}. However the 
minus sign (or more generally the normalization) of the local 
parameters $x, y$ is not a consequence of this compactification. This 
is the {\it right} normalization observed widely in \cite{HKTY} for 
the mirror map and also for the right predictions of Gromov-Witten 
invariants of complete intersection Calabi-Yau manifolds.  
\end{remark}

\vskip0.5cm

\noindent
{\bf (2-4)} {\it Picard-Fuchs equations of ${\mathcal X}^\vee$}: 
We may consider the symmetric reduction $a_i, b_i=c_i$ of the family 
simply by setting $x=y$. Then the Picard-Fuchs 
differential equation satisfied by the period integrals 
of the $\sigma^\vee$-quotient family ${\mathcal X}^\vee$ follows from 
Proposition \ref{prop:GKZxy} as follows: 

\begin{proposition} The mirror family ${\mathcal X}^\vee$ 
of the Reye congruence $X$ is defined over the diagonal $\mP^1$ of 
$x=y$. The period integrals of this family satisfies the Picard-Fuchs 
equation ${\mathcal D} \omega(x)=0$ with $\theta_x=x \frac{d \;}{dx}$ and 
\begin{equation}
\def\tx{\theta_x}
\begin{aligned}
{\mathcal D}=&
49 \,\tx^4-7x(155\tx^4+286\tx^3+234\tx^2+91\tx+14) \\
& 
-x^2(16105 \tx^4+680044 \tx^3+102261 \tx^2+66094 \tx +15736) \\
&
+8x^2(2625\tx^4+8589\tx^3+9071\tx^2+3759\tx+476)\\
&
-16x^4(465\tx^4+1266\tx^3+1439\tx^2+806\tx+184) 
+512 x^5 (\tx+1)^4 .
\end{aligned}
\mylabel{eqn:PFx}
\end{equation}
\end{proposition} 

The above differential equation appeared first in \cite{BaS} for the 
mirror family of $\tilde X$ in the early stage of mirror symmetry. We 
should note, however, that we are considering the same equation for 
the mirror family of our Reye congruence $X$, i.e., with the involution 
$\sigma^\vee$. This difference of interpretation becomes crucial when 
calculating higher-genus Gromov-Witten invariants in Sect.(2-6.2) and 
(2-6.3).

From the form of the Picard-Fuchs equation (\ref{eqn:PFx}), we observe that 
$x=0$ is a singular point with maximally unipotent monodromy, which is 
equivalent to the LCSL point for one dimensional families \cite{Mo}. 
We also observe that $x=\infty$ has the same property. Our interest 
hereafter will be to reveal a Calabi-Yau geometry which comes naturally 
with our Reye congruence and appears near this latter boundary point. 
Similar phenomena has been observed first by R{\o}dland for a Calabi-Yau 
complete intersection in a Grassmannian and a certain Pfaffian variety 
extracting genus zero Gromov-Witten invariants \cite{Ro}. 
In case of the Grassmannian and Pfaffian Calabi-Yau manifolds, higher genus 
Gromov-Witten invariants have been calculated recently 
from the Picard-Fuchs equation \cite{HK}. After a summary of the 
properties of our Picard-Fuchs equation 
(\ref{eqn:PFx}), we will determine the higher genus Gromov-Witten 
invariants which come out from the two different LCSL points.

\vskip0.5cm

\noindent
{\bf (2-5)} {\it Monodromy calculations}:   
The Picard-Fuchs equation is a differential equation of Fuchs type with 
its all singularities being regular. The entire picture can be grasped by 
the following Riemann's ${\mathcal P}$-symbol:
\begin{equation}
{\small 
\left\{
\begin{array}{c|cccccc}
x     & 0 & 1/32 & \alpha_1 & \alpha_2 & 7/4 & \infty \\
\hline
\rho_1 & 0 &    0    &   0    &   0    &  0  &  1      \\
\rho_2 & 0 &    1    &   1    &   1    &  1  &  1      \\
\rho_3 & 0 &    1    &   1    &   1    &  3  &  1      \\
\rho_4 & 0 &    2    &   2    &   2    &  4  &  1      \\
\end{array}
\right\} } \;\;, 
\mylabel{eqn:P-indices}
\end{equation}
where $\alpha_k$ are the roots of the equation $1+11 x - x^2=0$. 
The singularities at $1/32, \alpha_1=(11-5\sqrt{5})/2, \alpha_2=
(11+5\sqrt{5})/2$ are called {\it conifold} 
in general (although $1/32$ is slightly different from the latter two, 
see Proposition \ref{prop:odp} below). 
Whereas at the point $x=7/4$, called apparent singularity, there is no 
monodromy around the point. 

\begin{proposition} \mylabel{prop:odp}
The fiber of the mirror family ${\mathcal X}^\vee$ over $\mP^1$ 
has ordinary double points over $x=\alpha_1, \alpha_2$, and a $\mZ_2$ 
quotient of the ordinary double point over $x=1/32$. 
\end{proposition}

\begin{proof}
From the Picard-Fuchs equation, we see that the singularities at the 
boundaries are maximally unipotent. For others, we study the Jacobian 
ideal of (\ref{eqn:defeq-fa}):  
$a+U_i+V_i=0 \; (i=1,\cdots,5), U_1\cdots U_5=V_1\cdots V_5=1$.  
By computing a Gr\"obner basis with a suitable term order, 
we see that it consists of 
equations $V_1=V_2=V_3=V_4=V_5$, $a+U_i+V_i=0 \,(i=1,\cdots,5)$, $V_5^5-1=0$ 
and five other higher order equations of $V_5$ and $a$, which result 
in the discriminant when $V_5$ is eliminated.   
The defining equations are invariant under 
$(U_i,V_i,a) \rightarrow (\lambda U_i, \lambda V_i, \lambda a)$ for 
$\lambda^5=1$, and so is the Jacobian ideal. With this $\mZ_5$ action, 
one may summarize all the zeros of the Jacobian ideal into the following 
orbits:
\[
\mZ_5 \cdot (\text{-1-}a, \, \text{-1-}a, \, \text{-1-}a, \, 
\text{-1-}a,\,\text{-1-}a, \, 1,1,1,1,1,a)
\]
for each solution of $(1+a)^5=-1$. It is easy to see that these zeros 
are ordinary double 
points in $X^\vee_a$. Hence each orbit represents one ordinary double point 
in (the isomorphism class of) $X^\vee_x$ parametrized by $x=-\frac{1}{a^5}$. 
Thus we have five ordinary double points corresponding to each solution 
of $(1+a)^5=-1$. We see that the five solutions are mapped to $x= 
\alpha_1,\alpha_2, 1/32$, with two ordinary double points being identified 
under the involution $\sigma^\vee$ for each $\alpha_1, \alpha_2$ 
and making a $\mZ_2$ quotient over $x=1/32$. 
\end{proof}

Since the Picard-Fuchs equation (\ref{eqn:PFx}) is not of hypergeometric 
type, there is little hope to describe its integral basis 
for the solutions analytically. However, we may put several ansatz 
and requirements coming from mirror symmetry. 
Combined these with numerical calculations 
which are available for example in Maple \cite{Ma}, 
we can arrive at an integral basis.  

\vskip0.3cm
\noindent
(2-5.1) Near $x=0$, there exists a unique regular solution up to 
normalization. We fix the solution by a closed formula that come 
from the GKZ hypergeometric series:
\[
w_0(x)=\sum_{n,m \geq 0} \frac{((n+m)!)^5}{(n!)^5(m!)^5}x^n y^n |_{x=y} = 
1+2x+34x^2+\cdots .
\]
\noindent
(2-5.2) All other solutions contain logarithmic singularities, which may be 
explained by the Frobenius method applied to GKZ system \cite{HKTY}. We fix 
these solutions by requiring the following form:
\[
\begin{matrix}
w_1(x)= w_0(x) \log x + w_1^{reg}(x) , \hfill \\ 
w_2(x)= -w_0(x) (\log x )^2 + 
2 w_1(x) \log x + w_2^{reg}(x), \hfill  \\
w_3(x)=w_0(x) (\log x )^3 - 3 w_1(x) (\log x )^2 + 3 w_2(x) \log x 
+ w_3^{reg}(x),  \hfill \\
\end{matrix}
\]
where $w_k^{reg}(x)$ represents the regular part of the solution $w_k(x)$, and 
we require that the series expansion of $w_k^{reg}(x)/w_0(x)$ does not contain a constant term 
for each $k=1,2,3$. 

\noindent
(2-5.3) Define the mirror map, $x=x(q)$, by inverting the 
series 
\[q=e^{\frac{w_1(x)}{w_0(x)}}=x(1+5 x + 90 x^2 + \cdots) .
\]

\noindent
(2-5.4) Local solutions around $x=\infty$ may be arranged in the same 
way above.  Setting $z=\frac{1}{x}$, we normalize  the regular solution
\[
\tilde w_0(z)=2 \sum_{n\geq 1} \,_5F_4(n^5,1^4;-1) \, z^n = 
z
+\frac{1}{2}z^2
+\frac{227}{128} z^3
+\frac{4849}{512} z^4 \cdots ,
\]
where $\,_5F_4(n^5,1^4;x)$ is the generalized hypergeometric series. 
For the other solutions, similarly to (2-5.2), we set the following ansatz:
\[
\begin{matrix}
\tilde w_1(z)= \tilde w_0(z) \log c z + \tilde w_1^{reg}(z) , \hfill \\ 
\tilde w_2(z)= -\tilde w_0(z) (\log c z )^2 + 
2 \tilde w_1(z) \log c z + \tilde w_2^{reg}(z), \hfill  \\
\tilde w_3(z)=\tilde w_0(z) (\log c z )^3 - 3 \tilde w_1(z) (\log c z )^2 + 
3 \tilde w_2(z) \log c z 
+ \tilde w_3^{reg}(z),  \hfill \\
\end{matrix}
\]
with some constant $c$, and require that the series expansion of 
$\tilde w_k^{reg}(z)/\tilde w_0(z)$ does not contain a constant term 
for each $k$. 

\noindent
(2-5.5) Denote by $\Omega_x$ a holomorphic 
three form on the mirror Calabi-Yau threefold over a point 
$x \in \mP^1$ of the family ${\mathcal X}^\vee$. 
Fix a symplectic basis $\{ A_0, A_1, B^1, B^0\}$ of 
$H^3(X^\vee_{x_0},\mZ)$ with its symplectic form 
$
{\mathsf S}=\left( \begin{smallmatrix} 
0 & 0 & 0 & 1 \\ 
0 & 0 & 1 & 0 \\ 
0 & -1 & 0 & 0 \\ 
-1 & 0 & 0 & 0  \end{smallmatrix} \right)
$, and define the Period integrals; $\,^t \Pi(x) = 
( \int_{A_0} \Omega_x, \int_{A_1} \Omega_x, \int_{B^1} \Omega_x, 
\int_{B^0} \Omega_x )$. 
Using the mirror symmetry which arises 
near LCSL \cite{CdOGP}, we make the following ansatz for the period 
integrals in terms of the local solutions: 
\[
\Pi(x) = \left( 
\begin{smallmatrix} 1 & 0 & 0 & 0 \\
0 & 1 & 0 & 0 \\
\beta & a & \kappa/2 & 0 \\
\gamma & \beta & 0 & 
- \kappa/6 \\ 
\end{smallmatrix} \right)
\left( \begin{smallmatrix} 
n_0 w_0(x) \\  
n_1 w_1(x) \\ 
n_2 w_2(x) \\ 
n_3 w_3(x) \end{smallmatrix} \right)\;,\;
\tilde\Pi(z) = N_z \left( 
\begin{smallmatrix} 1 & 0 & 0 & 0 \\
0 & 1 & 0 & 0 \\
\tilde\beta & \tilde a & \tilde \kappa/2 & 0 \\
\tilde \gamma & \tilde \beta & 0 & 
- \tilde \kappa/6 \\ 
\end{smallmatrix} \right)
\left( \begin{smallmatrix} 
n_0 \tilde w_0(z) \\ 
n_1 \tilde w_1(z) \\ 
n_2 \tilde w_2(z) \\ 
n_3 \tilde w_3(z) \end{smallmatrix} \right)\;,\;
\]
where $\kappa=\deg(X), \beta=-\frac{c_2.H}{24}, 
\gamma=-\frac{\zeta(3)}{(2\pi i)^3} e(X)$ with the topological 
invariants of the Reye congruence, i.e., $\deg(X)=35, c_2.H=50, 
e(X)=-50$ and $a$ is an unknown parameter with no geometric 
interpretation, also we set $n_k=\frac{1}{(2\pi i)^k}$. 
$\tilde \Pi(z)$ is supposed to be a symplectic 
transform of $\Pi(x)$ and for the parameters 
$\tilde \kappa, \tilde \beta, \tilde \gamma$, similar geometric 
interpretations are expected for the solutions about $z=0$. 

\noindent
(2-5.6) 
With a choice of the holomorphic three form, the Griffith-Yukawa 
coupling is defined by  
$C_{xxx}:= - \int_{X^{\vee}_x} \Omega_x \wedge\frac{d^3 \;}{dx^3} \Omega_x$. 
By the fact that period integrals satisfy the Picard-Fuchs equation 
(\ref{eqn:PFx}), we can determine it up to some constant \cite{CdOGP}, 
\[
C_{xxx}=-\,^t\Pi(x) \,{\mathsf S}\, \big(\frac{d^3\;}{dx^3} \Pi(x) \big)= 
\frac{K(35-20 x)}{x^3 (1-32 x)(1+11x-x^2)},  
\]
where the constant $K$ will be fixed to $1$ later to have a right 
normalization of the quantum cohomology ring of $X$. Also we have 
the following relations expressing the Griffiths transversality: 
\[
\,^t\Pi(x) \,{\mathsf S}\, \big(\frac{d\;}{dx} \Pi(x) \big) = 
\,^t\Pi(x) \,{\mathsf S}\, \big(\frac{d^2\;}{dx^2} \Pi(x) \big) =0 .
\]

\noindent
(2-5.7) The relations in (2-5.6) provide conditions for the ansatz 
of period integrals in (2-5.5). In addition to these, we should have 
another constraint;
\[
\,^t\Pi(x) \,{\mathsf S}\, \big(\frac{d^3\;}{dx^3} \Pi(x) \big) (-x^2)^3 
=\,^t \tilde \Pi(z) \,{\mathsf S}\, 
\big(\frac{d^3\;}{dz^3} \tilde \Pi(z) \big), 
\]
which expresses the relation $C_{xxx} (\frac{dx}{dx})^3 = C_{zzz}$.

We have obtained the following results once 
passing to a numerical calculations (see \cite{EnS} for example).

\begin{proposition} \mylabel{prop:anzatz}
\begin{enumerate}[(1)]
\item 
There exists a solution for the ansatz (2-5.5) of integral, symplectic 
basis of the solutions of Picard-Fuchs equation (\ref{eqn:PFx}) with 
\[
\tilde \kappa =10 \;,\; 
\tilde \beta = -\frac{40}{24} \;,\; 
\tilde \gamma =-\frac{\zeta(3)}{(2\pi i)^3} (-50) \;,\; 
c=\frac{1}{2^5}\;,\; a=-\frac{1}{2}\;,\; \tilde a=0 \;,\; N_z= \frac{1}{4}.
\]
\item
When the analytic continuation is performed along a path in the upper half 
plane, the two local solutions are related by 
a symplectic matrix $\Pi(x)=S_{xz} \tilde \Pi(x)$ with 
$S_{xz}=\left( \begin{smallmatrix} 
-2 & 5 & -1 & 4 \\
0 & 2 & 0 & 1 \\
5 & -1 & 3 & -3 \\
0 & -5 & 0 & -3 \\ \end{smallmatrix} \right)$.  
\item
Monodromy matrices at each singular point have the forms given in 
Table 1. 
\end{enumerate}
\end{proposition}

\def\m{\text{{\bf -}}}
{\tablinesep=1.5pt
\tabcolsep=1pt 
\begin{tabular}{|c|c|c|c|c|c|}
\hline
&{\small $x=\alpha_1$} & {\small 0} & {\small 1/32} & 
{\small $\alpha_2$} & {\small $\infty$} \\
\hline 
{\small $\Pi(x)$} & $\left(\begin{smallmatrix}
11 & \m14 & 2 & 4 \\
5 & \m6 & 1 & 2 \\
35 & \m49 & 8 & 14 \\
\m25 & 35 & \m5 & \m9 \\ \end{smallmatrix} \right)$
&
$\left(\begin{smallmatrix}1 & 0 & 0 & 0 \\
                  1 & 1 & 0 & 0 \\
                  17 & 35 & 1 & 0 \\
                  \m10 & \m18 & \m1 & 1  \end{smallmatrix} \right)$
&
$ \left(
                 \begin{smallmatrix}
                  1 & 0 & 0 & 2 \\
                  0 & 1 & 0 & 0 \\
                  0 & 0 & 1 & 0 \\
                  0 & 0 & 0 & 1
                 \end{smallmatrix}
                 \right)$
&
$\left(
                  \begin{smallmatrix}
                   1 & 10 & 0 & 4 \\
                   0 & 1 & 0 & 0 \\
                   0 & \m25 & 1 & \m10 \\
                   0 & 0 & 0 & 1
                  \end{smallmatrix}
                  \right)$
&
$\left(
                  \begin{smallmatrix}
                   41 & \m34 & 12 & 30 \\
                   4 & \m6 & 1 & 2 \\
                   \m72 & 51 & \m22 & \m56 \\
                   \m15 & 18 & \m4 & \m9
                  \end{smallmatrix}
                  \right)  $
\\
\hline
{\small $\tilde\Pi(z)$} & $\left(
                  \begin{smallmatrix}
                   \m19 & 0 & \m8 & 16 \\
                   10 & 1 & 4 & \m8 \\
                   0 & 0 & 1 & 0 \\
                   \m25 & 0 & \m10 & 21
                  \end{smallmatrix}
                  \right)$
&
$ \left(
                  \begin{smallmatrix}
                   \m19 & 170 & \m8 & 105 \\
                   9 & \m69 & 4 & \m43 \\
                   5 & \m110 & 1 & \m65 \\
                   \m20 & 145 & \m9 & 91
                  \end{smallmatrix}
                  \right)$
&
$ \left(
                  \begin{smallmatrix}
                   1 & 30 & 0 & 18 \\
                   0 & 1 & 0 & 0 \\
                   0 & \m50 & 1 & \m30 \\
                   0 & 0 & 0 & 1
                  \end{smallmatrix}
                  \right)$
&
$ \left(
                  \begin{smallmatrix}
                   1 & 0 & 0 & 1 \\
                   0 & 1 & 0 & 0 \\
                   0 & 0 & 1 & 0 \\
                   0 & 0 & 0 & 1
                  \end{smallmatrix}
                  \right)$
&
$\left(
                  \begin{smallmatrix}
                   1 & 0 & 0 & 0 \\
                   1 & 1 & 0 & 0 \\
                   5 & 10 & 1 & 0 \\
                   \m5 & \m5 & \m1 & 1
                  \end{smallmatrix}
                  \right)$  \\
\hline
\end{tabular} 
} 

\centerline{ {\bf Table 1.} Monodromy matrices. }

\vskip0.3cm
Fitting the period integral $\tilde \Pi(z)$ into the 
form (2-5.5) which admit mirror interpretation, we come to the following 
conjecture for the mirror geometry that emerges about the second LCSL point 
at $x=\infty$:

\vskip0.3cm
\noindent
{\bf Conjecture 1.} {\it 
There exists a smooth projective Calabi-Yau threefold $Y$ with 
its topological invariants
\[
deg(Y)=10\;,\;\; c_2.H = 40 \;,\;\; e(Y)=-50 \;, \;\; h^{1,1}(Y)=1, 
\;\; h^{2,1}(Y)=26,
\]
where $H$ is the ample generator of the Picard group $Pic(Y) \cong \mZ$. 
}

\vskip0.3cm

The existence of $Y$ may be expected from the table  
of `Calabi-Yau differential equations` of fourth order in \cite{EnS}. 
Here we have arrived at Conjecture 1 starting from the Reye congruence 
$X$.  One should note that $Y$ is not birational to $X$ since 
$\rho(X)=\rho(Y)=1$ and $deg(X)\not= deg(Y)$. 
The appearring relation between the two is reminiscent of 
the example of Calabi-Yau threefolds related to Grassmannian and 
Pfaffian \cite{Ro}\cite{BCa}, 
which has been understood in the proposal 
`homological projective duality` \cite{Ku1}. 
In this context, we naturally come to:

\vskip0.2cm
\noindent
{\bf Conjecture 2.} {\it The Reye congruence $X$ and $Y$ have 
equivalent derived categories 
of coherent sheaves; $D(Coh(X))$ $\cong$ $D(Coh(Y))$. }

\vskip0.3cm

We will prove Conjecture 1 in the following section while 
we have to leave Conjecture 2 for future investigations. 
Here we remark that analogous conjectures may be stated for 
odd-dimensional Reye congruences in general, since one can observe 
that Picard-Fuchs equations have similar properties in odd-dimensions.

\vskip0.5cm

\noindent
{\bf (2-6) } {\it Gromov-Witten invariants}:  
One of the important applications of the mirror symmetry is 
predicting   
Gromov-Witten invariants of Calabi-Yau manifolds \cite{CdOGP}, \cite{BCOV}. 
Combined with  
Conjecture 1, we can extract Gromov-Witten invariants 
for $X$ and $Y$ from the LCSL point at $x=0$ and $z=0$, respectively. 

\noindent
(2-6.1) The genus 0 Gromov-Witten invariants of the Reye 
congruence $X$ can be read from the $q$ expansions ($q:=e^t$) of the 
Griffiths-Yukawa coupling,
\[
\frac{1}{w_0(x)^2} C_{xxx} \Big(\frac{d x}{dt}\Big)^3 = 35+ 
\textstyle{\sum_{d \geq 1}} \, d^3 \,  
N_{0}^X(d) \, q^d \;,
\]
with $K=1$ in $C_{xxx}$ to have $35=\deg(X)$ at the constant term. 
For the conjectural geometry $Y$, we define the mirror map 
$z=z(\tilde q)$ by inverting the series 
$
\tilde q=e^{\frac{\tilde w_1(z)}{\tilde w_0(z)}} = c z ( 
z+\frac{35 z^2}{16}+\frac{10395 z^3}{1024}+ \cdots ),
$
with $c=\frac{1}{2^5}$. Then genus 0 Gromov-Witten invariants of $Y$ 
are read from 
\[
\frac{1}{\tilde w_0(z)^2} C_{zzz} \Big(\frac{d z}{d\tilde t} \Big)^3 = 
10+ \textstyle{\sum_{d \geq 1}} \, d^3 \, N_0^Y(d) \, \tilde q^d \;,
\]
with $K=1$ in $C_{zzz}=C_{xxx}(\frac{dx}{dz})^3$ as fixed above.  
The genus one Gromov-Witten invariants follow similarly by using the 
BCOV formula in \cite{BCOV} with the topological invariants of $X$ and 
those given in Conjecture 1 for $Y$. 

\noindent
(2-6.2) Extracting higher genus ($g \geq 2$) Gromov-Witten invariants 
from period integrals becomes more involved than above. However 
calculations are essentially the same as those formulated 
in the Grassmannian and Pfaffian (\cite{HK} and references therein). 
We simply present the 
resulting BPS numbers, which are determined from 
Gromov-Witten invariants (see Appendix A). 

\noindent
(2-6.3) We can also determine Gromov-Witten invariants for the covering 
$\tilde X$ (or $\tilde X_0$). Higher genus calculations in this case 
are more complicated and time consuming. The details will be reported 
elsewhere, and we simply list the resulting BPS numbers $g=0,1,2$ in 
Appendix B.

\vskip1cm

\section{{\bf Projective duality and the double covering}}

\noindent
{\bf (3-1)} {\it Mukai dual of $\tilde X_0$}:
Here we interpret the three LCSL points observed in the mirror 
family $\tilde {\mathcal X}_0^\vee$ to the Calabi-Yau manifold $\tilde X_0$. 

Recall that $\tilde X_0$ is defined as a complete intersection of 
five generic 
global sections of $H^0(\mP^4 \times \mP^4, {\mathcal O}(1,1))$. 
Explicitly one may write the defining equations as
\[
\,^t z {\tt A}_1 w = 
\,^t z {\tt A}_2 w = 
\,^t z {\tt A}_3 w = 
\,^t z {\tt A}_4 w = 
\,^t z {\tt A}_5 w = 0 ,
\]
where ${\tt A}_i$ are general $5 \times 5$ ${\bf C}$-matrices, which 
generalize the corresponding symmetric matrices $A_i$ in Sect.(2-1). 
Consider the Segre embedding $\mP^4 \times \mP^4 \hookrightarrow 
\mP^{24}= {\rm Proj} {\bf C}[x_{ij}| 1 \leq i,j \leq 5]$ 
by $x_{ij}=z_i w_j$.  
Then the global sections of the $(1,1)$ divisor extend 
to linear forms $H_i \, (i=1,\cdots,5)$ on $\mP^{24}$, and 
we have
\[
\tilde X_0 = (\mP^4 \times \mP^4) \cap H_1 \cap \cdots \cap H_5 
\; \subset \mP^{24} \;.
\]
Denote by $\vP^{24}:=(\mP^{24})^*$ the dual projective space to $\mP^{24}$ and 
by $(\mP^4 \times \mP^4)^*$ the dual variety in $\vP^{24}$ 
to the Segre embedding in $\mP^{24}$. 
According to Mukai, we define a modified dual of $\tilde X_0$ 
to be 
\[
\tilde X_0^\sharp := (\mP^4 \times \mP^4)^* \cap 
\langle h_1, \cdots, h_5 \rangle \subset \vP^{24} \;,
\]
where $\langle h_1, \cdots,h_5 \rangle$ represents the projective space 
spanned by the dual points $h_i$ to $H_i$.  
We call this the {\it Mukai dual} to $\tilde X_0$. 
The following  is a classical result:

\begin{lemma} \mylabel{lemma:segreD}
$(\mP^4 \times \mP^4)^*=
\{ [\vc_{ij} ] \in \vP^{24} \,|\,  \det(\vc_{ij})=0 \}$.
\end{lemma}

\begin{proof}
Consider a hyperplane ${\mathcal H}:=\sum_{ij} \vc_{ij}x_{ij}=0$ 
in $\mP^{24}$. Denote by $D$ the restriction of ${\mathcal H}$ 
to $\mP^4 \times \mP^4 \subset \mP^{24}$. The dual variety 
consists of the points $[\vc_{ij}]$ in $\vP^{24}$ for 
which ${\mathcal H}$ is tangent to $\mP^4 \times \mP^4$, i.e., 
$D$ is singular in $\mP^4 \times \mP^4$. Using the equation 
$\sum \vc_{ij} z_i w_j =0$ 
of $D$, and the Jacobian criterion, it is straightforward to see 
that the condition $\det(\vc_{ij})=0$ is equivalent for $D$ to be singular. 
\end{proof}

It is well-known that $\Sing (\mP^4 \times \mP^4)^*$
is the locus of $5\times 5$ matrices of rank less than or equal to
three in $\vP^{24}$.

\begin{lemma}
\mylabel{lemma:tangent}
$\deg \Sing (\mP^4 \times \mP^4)^*=50$. 
Let ${\tt A}$ be a $5\times 5$ matrix
of rank three.
Then, analytically locally near $[{\tt A}]$,
$(\mP^4 \times \mP^4)^*$ is isomorphic to the product of 
the $3$-dimensional ordinary double point and $\mC^{20}$.
In particular, $\codim \Sing (\mP^4 \times \mP^4)^*=3$.
\end{lemma}

\begin{proof}
The first statement is a special case of \cite[Proposition 12]{HTu}.
The second statement can be proved in a similar way to
\cite[Chap. 2, \S 3, Lemma 2.3]{Ty}.
We include a proof for the readers' convenience.
We have only to determine the tangent cone of $(\mP^4 \times \mP^4)^*$ 
at $[{\tt A}]$.
We may assume that 
$\,^tz {\tt A} w =z_1w_1+z_2w_2+z_3w_3$.
Let ${\tt B}=(b_{ij})$ be a $5\times 5$ matrix.
The line ${\tt A}+t{\tt B}$ in $\vP^{24}$
is contained in the tangent cone of $(\mP^4 \times \mP^4)^*$ at $[{\tt A}]$
if and only if the degree two term of $\det ({\tt A}+t{\tt B})$ vanishes,
equivalently,
$\det \begin{vmatrix} b_{44} & b_{45}\\
                      b_{54} & b_{55}\\
\end{vmatrix}=0$.
This implies that the tangent cone of $(\mP^4 \times \mP^4)^*$ 
at $[{\tt A}]$
is the cone over the smooth quadric surface $\{b_{44}b_{55}-b_{45}b_{54}=0\}$
in $\mP^3$
with vertex $\mP^{20}$.
\end{proof}

\begin{proposition}  
The dual variety 
$\tilde X_0^\sharp$ is a determinantal quintic: 
$\det ( \sum_{i=1}^5 y_i {\tt A}_i ) =0$ in 
$\vP^4_{h}:=\langle h_1, \cdots,h_5 \rangle$.  
$\Sing \tilde X_0^\sharp$ consists of $50$ ordinary double points. 
\end{proposition}

\begin{proof}
Since the Segre embedding is defined by $x_{ij}=z_iw_j$, the linear 
forms may be written 
$H_k=\,^tz {\tt A}_k w = \sum_{i,j} {\tt a}^{k}_{ij} x_{ij}$ with 
${\tt A}_k=({\tt a}^k_{ij})$. Then the dual points to the hyperplanes 
are given by 
$h_k=[{\tt a}^k_{ij}] \in \vP^{24}$. Therefore the elements of 
$\vP^4_{h}:=\langle h_1, \cdots,h_5 \rangle$ are written by 
$\sum_k y_k [{\tt a}^k_{ij}]$. Then the first statement follows from the 
Lemma \ref{lemma:segreD}. For a general choice of 
${\tt A}_i$, $\tilde X_0^\sharp$ is a general linear section of
$(\mP^4 \times \mP^4)^*$.
Thus, by Lemma \ref{lemma:tangent},
the last statement follows.
\end{proof}

\noindent
{\bf (3-2)} {\it The three LCSL points in Proposition \ref{prop:GKZxy}}: 
Define matrices ${\tt B}_i, {\tt C}_i$ by the following relations:
\[
\sum y_i {\tt A}_i =
\,^t
({\tt B}_1 y \; {\tt B}_2 y  \; {\tt B}_3 y  \; 
 {\tt B}_4 y  \; {\tt B}_5 y) =
 ({\tt C}_1 y \; {\tt C}_2 y \; {\tt C}_3 y \; {\tt C}_4 y \; {\tt C}_5 y ), 
 \]
where $y=\,^t(y_1\,y_2\, \cdots y_5)$. 
The $5\times 5$ matrices ${\tt B}_i, {\tt C}_i$ 
and ${\tt A}_i$ are different in general. 
Define Calabi-Yau complete intersections $\tilde X_1, \tilde X_2$ 
in $\mP^4 \times \mP^4$ by
\[
\begin{aligned}
& \tilde X_1 =\{(y, w)\mid (\sum y_i {\tt A}_i)w=0\}
=\{(y,w)\mid \,^t y \,^t {\tt B}_i w=0\, (1\leq i \leq 5) \, \}, \\
&\tilde X_2 =\{(z, y)\mid \,^t z (\sum y_i {\tt A}_i)=0 \}
=\{(z,y) \,\mid \,^t z {\tt C}_i y=0\, (1\leq i \leq 5) \, \}.
\end{aligned}
\]
We also set
\[
Z_1=\{ w \in \mP^4 \;|\; \det \, ( {\tt A}_1w \cdots {\tt A}_5w)=0 \;\}
\,,\;
Z_2=
\{ z \in \mP^4 \;|\; \det  (\,^t z {\tt A}_1 \cdots \,^t z {\tt A}_5)=0 \;\}.
\]
Then we see that there exists the following diagram:
\begin{equation}
\mylabel{eqn:diagram0}
\begin{matrix}
\xymatrix{
& \tilde X_1 \ar[dl]\ar[dr] &\dashleftarrow & 
\tilde{X}_0 \ar[dl]\ar[dr]&\dashrightarrow & 
\tilde{X}_2\ar[dl]\ar[dr] & &\\
\tilde X_0^\sharp &  & Z_1  &  & Z_2 & &  \tilde X_0^\sharp &}
\end{matrix} 
\end{equation}
Both $\tilde X_1$ and $\tilde X_2$ are smooth for generic 
choices of ${\tt A}_i$. It is easy to see that 
$\det \, \sum y_i {\tt A}_i=0$ for $(y,w) \in \tilde X_1$. 
Hence by the correspondence $(y,w) \mapsto  
\sum y_i {\tt A}_i$, we have 
a map $\tilde X_1 \rightarrow \tilde X_0^\sharp$.
Since $rk \sum y_i {\tt A}_i
=4$ for a smooth point $y$ of $\tilde X_0^\sharp$, the map 
$\tilde X_1 \rightarrow \tilde X_0^\sharp$ is 
bijective except the 50 singular points. 
Therefore 
$\tilde X_1 \rightarrow \tilde X_0^\sharp$ is a resolution,
which is crepant since both the canonical bundles of
$\tilde X_1$ and $\tilde X_0^\sharp$ are trivial.  
Since $\tilde X_1 \rightarrow \tilde X_0^\sharp$ is crepant,
it is a small resolution of $50$ ordinary double points.
Similarly, we can define a map $\tilde X_2 \rightarrow \tilde X_0^\sharp$
by the correspondence $(z,y) \mapsto  \sum y_i {\tt A}_i$,
which is another crepant resolution of $\tilde X_0^\sharp$.

By the natural projections, we have 
maps $\tilde X_i\to Z_i$ and $\tilde X_0\to Z_i$ $(i=1,2)$,
all of which are crepant resolutions.
Since $Z_1$ and $Z_2$ are also generic determinantal quintics, they have
$50$ ordinary double points respectively by Lemma \ref{lemma:tangent}, 
and then
$\tilde X_i\to Z_i$ and $\tilde X_0\to Z_i$ $(i=1,2)$ are small resolutions. 

In general, $\tilde X_i$ and $\tilde X_0$ are not
isomorphic over $Z_i$ ($i=1,2$). 
Indeed, if $\tilde X_1$ and $\tilde X_0$ were 
isomorphic over $Z_1$ for example,
then $\tilde X_0^\sharp$ and $Z_2$ would be isomorphic but this is impossible
for a general choice of ${\tt A}_i$. 
Hence $\tilde X_0\dashrightarrow \tilde X_i$ ($i=1,2$) are the Atiyah flops
since all the Picard numbers of $\tilde X_i$ and $\tilde X_0$
are two. Thus all of $\tilde X_i (i=0,1,2)$ are smooth Calabi-Yau 
$3$-folds which are birational to each other.  In particular, they 
have the equivalent derived categories due to \cite{Br}.

Note that $\tilde X_0, \tilde X_1$ and $\tilde X_2$ are all complete 
intersections of five $(1,1)$-divisors in $\mP^4\times\mP^4$, 
and thus in the same deformation 
family. Then, by the Batyrev-Borisov mirror construction, we see that 
they share the same mirror family $\tilde {\mathcal X}^\vee_0$ over 
$\mP^2$, where we have found three LCSL points (Proposition 
\ref{prop:GKZxy}).  One may naturally consider that the geometry  
of $\tilde X_i$ $(i=0,1,2)$ appears near each LCSL point under the 
mirror symmetry. This is reminiscent of the topology change studied 
in \cite{AGM}, however, one should note that our flops do not come 
from those of the ambient space.

\vskip0.3cm

\noindent
{\bf (3-3)} {\it Sym$^2 \mP^4$ and projective duality}: 
Consider the Chow variety ${\rm Chow}^2 \mP^4$ of $0$-cycles of 
degree 2 in $\mP^4=\mP(\mC^5)$. 
A Chow variety, in general, may be embedded into a projective variety which 
is defined by the so-called Chow form. Let $x+y$ be a $0$-cycle of degree 
2 in $\mP^4$. The Chow form in this case is given by the product 
of two linear forms $(x\cdot \xi)(y \cdot \xi)$ with $\xi \in \vC^5$. 
Then the embedding ${\rm Chow}^2 \mP^4 \hookrightarrow 
\mP({\rm Sym}^2 \mC^5)$ 
is defined by $x+y \mapsto [p_{ij}(x,y)]$ with  
$(x\cdot \xi)(y \cdot \xi) = \sum_{i \geq j} p_{ij}(x,y) 
\xi_i\xi_j$. Since 
$p_{ij}(x,y)=\frac{1}{2}(x_iy_j+x_jy_i) \;\;(1 \leq i,j \leq 5)$, we 
see that the embedding is given by the linear system of symmetric 
$(1,1)$-divisors on $\mP^4 \times \mP^4$.  Since the global 
sections of symmetric $(1,1)$-divisors generate symmetric polynomials 
in ${\bf C}[x_i,y_j]$, we see the isomorphism 
${\rm Chow}^2 \mP^4 \cong {\rm Sym}^2 \mP^4$ as an algebraic 
variety \cite{GKZ2}. 

Our Reye congruence $X$ (\ref{eqn:def-Reye}) is defined as a 
subvariety in the Grassmannian $G(2,5)$. On the other hand, the 
isomorphic quotient $\tilde X/\langle \sigma \rangle$ may be regarded 
as a subvariety in ${\rm Sym}^2 \mP^4$. We remark that these two are 
connected by the natural diagram: 
$G(2,5) \leftarrow {\rm Hilb}^2 \mP^4 \rightarrow {\rm Chow}^2 \mP^4  
\cong {\rm Sym}^2 \mP^4$ with the standard morphisms.

Let $\Sigma:={\rm Sym}^2 \mP^4$, and $\Sigma_0:=\Sing \Sigma$ be 
the singular locus of $\Sigma$. 
$\Sigma_0$ is the second Veronese variety $v_2(\mP^4)$,
namely, $\mP^4$ embedded in $\mP^{14}$ 
by the linear system of quadrics.

\begin{proposition}
$\Sigma$ is the secant variety of $\Sigma_0$,
i.e.,
$\Sigma=\overline{\cup \{\langle p, q\rangle\mid p, q 
\in \Sigma_0, p\not = q\}}$,
where $\langle p, q\rangle$ is the line through $p$ and $q$.
\end{proposition}

\begin{proof} Note the identity for the Chow embedding
\[
p_{ij}(x+y,x+y)-p_{ij}(x-y,x-y)=4 p_{ij}(x,y) .
\]
This implies the correspondence between 
the point $(x,y) \in {\rm Sym^2} \mP^4
\cong \Sigma$ $(x\not= y$) and the line 
$\langle x+y, x-y \rangle \in \Sigma$,
where we consider $x+y$ and 
$x-y \in v_2(\mP^4)=\Sigma_0$. 
\end{proof}

The projective dual $\Sigma_0^*$ is known to be the determinantal 
hypersurface in $(\mP^{14})^*=\mP({\rm Sym}^2 \vC^5)$. Further, the duality 
reverses the inclusion $\Sigma \supset \Sigma_0$ to 
$\Sigma^* \subset \Sigma_0^*$ with $\Sigma^*$ being the singular 
locus of $\Sigma_0^*$ consisting of $5\times 5$ matrices of rank $\leq 3$ 
 \cite[Chap.1, \S.4.C]{GKZ2}. 
Set ${\mathcal H}:=(\Sigma_0)^*$ and consider 
the correspondence \cite[Chap.3, \S.3]{Ty}:
\[
{\mathcal U}=\{ (x,A) \in \mP^4 \times {\mathcal H} \;|\; Ax=0 \;\} 
\subset \mP^4 \times \sH \;.
\]
Then ${\mathcal U}$ is a natural resolution of ${\mathcal H}$. To see the 
geometry of ${\mathcal U}$, denote the projection to the first and 
second factors by $\pi_1$ and $\pi_2$, respectively. Then the 
morphism $\pi_2:{\mathcal U} \rightarrow {\mathcal H}$
is one to one 
over ${\mathcal H}\setminus \Sigma^*$ since $A$ is of rank 4. Over a 
point of rank $i$, the fiber is isomorphic to $\mP^{4-i}$. 
The first projection $\pi_1$ to $\mP^{4}$ represents $\sU$ as a projective 
bundle over $\mP^4$, where the fiber $\sU_x$ over a point $x$  
is the space of singular quadrics in $\mP^4$ containing $x$ in their 
vertices. In particular, $\sU$ is smooth.
$\sU_x$ can be also regarded as the space of quadrics in 
$\mP^{3}$, where $\mP^{3}$ is the image of the projection 
$\mP^4\dashrightarrow \mP^{3}$ from $x$.

Let $\sM:=\pi_2^*\sO_{\sH}(1)$
and $\sL:=\pi_1^*\sO_{\mP^4}(1)$.
Denote by $\sE$ the $\pi_2$-exceptional divisor.
The divisor $\sE$ is contracted by $\pi_2$ to
the locus of symmetric matrices of rank $\leq 3$.
\begin{proposition}
\mylabel{prop:relation}
$\sE=4\sM-2\sL$.
\end{proposition}

\begin{proof}
This is a specialization of \cite[Chap. 3, \S 2, Lemma 3.1]{Ty}
and a generalization of \cite[Proposition 2.4.1]{Co}.
We repeat the proof of \cite[Chap. 3, \S 2, Lemma 3.1]{Ty} 
for the readers' convenience.

Since $\sL$ and $\sM$ freely generate $\Pic \sU$, 
we may write $\sE=x\sM-y\sL$ for some integers $x$ and $y$.
Recall that a fiber $F$ of $\pi_1\colon \sU\to \mP^4$ is
isomorphic to the space of quadrics in $\mP^3$,
and then $\sE|_F$ is identified with the space of
singular quadrics. Thus $\sE|_F$ is the determinantal hypersurface
of degree $4$ in $F\simeq \mP^9$.
Restricting $\sE=x\sM-y\sL$ to $F$ and noting $\sL|_F=0$ and
$\sM|_F=\sO_{\mP^9}(1)$,
we have $x=4$.

Let $Q\in \sH$ be a quadric of rank three.
Then the fiber $\gamma$ of $\pi_2\colon \sU\to \sH$ over $Q$
is isomorphic to $\mP^1$ and is mapped to the vertex of $Q$ by $\pi_1$.
By the following lemma, we have $\sE\cdot \gamma=-2$.
Restricting $\sE=x\sM-y\sL$ to $\gamma$ and noting 
$\sL|_{\gamma}=\sO_{\mP^1}(1)$ and
$\sM|_{\gamma}=0$,
we have $y=2$. 
\end{proof}
\begin{lemma}
Let $Q\in \sH$ be a quadric of rank three.
Then, analytically locally near $Q$,
$\sH$ is isomorphic to the product of 
the $2$-dimensional ordinary double point and $\mC^{11}$.
\end{lemma}

\begin{proof}
This is \cite[Chap. 2, \S 3, Lemma 2.3]{Ty}.
The proof is similar to that of Lemma \ref{lemma:tangent}.
In that proof, we have only to take as ${\tt B}$ a symmetric matrix.
\end{proof}

\noindent
{\bf (3-4)} {\it Hessian quintic of the Reye congruence $X$}:
Since the Chow embedding $\Sigma \hookrightarrow \mP({\rm Sym}^2\mC^5)$ is 
defined by symmetric $(1,1)$-divisors, we have for $X\cong \tilde X/\langle 
\sigma \rangle$: 
\[
X=\Sigma \cap H_1 \cap \cdots \cap H_5  \; \subset 
\mP({\rm Sym}^2\mC^5),
\]
where $H_i$ are linear forms on $\mP({\rm Sym}^2\mC^5)$ representing 
the quadratic forms $Q_i$ on $\vP^4$. The Mukai dual of $X$ has dimension 
one. Since we expect a threefold related to $X$, we shift by one 
in the inclusion $\Sigma_0 \subset \Sigma$ and define the {\it shifted} 
Mukai dual by  
\[
H = (\Sigma_0)^* \cap \langle h_1, \cdots, h_5 \rangle  \;
\subset \mP({\rm Sym}^2 \vC^5),
\]
where $(\Sigma_0)^*={\mathcal H}$ is the determinantal hypersurface.  
In our context, this definition may be regarded as the symmetric 
limit of the Mukai dual $\tilde X_0^\sharp$. 

By construction, the projective space $\mP^4_h :=
\langle h_1, \cdots, h_5 \rangle$ is nothing but the linear system $P$ of 
the quadrics. The shifted Mukai dual $H$ corresponds to  
{\it the Hessian surface} in the classical Reye congruence \cite{Co}. 
Explicitly $H$ may be written by
\[
H=\{ [y] \in \mP^4 \,|\, \det( \sum  y_i A_i) =0 \,\}, \;\;
\]
where we recall that $A_i$ $(i=1,\dots, 5)$ are bases of $P$.

We may now construct a diagram similar to (\ref{eqn:diagram0}).
Since $A_i$ are symmetric, we have the matrices ${B}_i$ 
satisfying 
\[
\sum y_i {A}_i =
\,^t ({B}_1 y \; {B}_2 y \; {B}_3 y \; {B}_4 y \; {B}_5 y ) =
({B}_1 y \; {B}_2 y \; {B}_3 y \; {B}_4 y \; {B}_5 y ), 
 \]
where $y=\,^t(y_1\,y_2\, \cdots y_5)$. 
Similarly to $\tilde X_1$ and $\tilde X_2$ in the diagram 
(\ref{eqn:diagram0}), we define
\[
\begin{aligned}
&U_1=\{ (y,w) \in H\times \mP^4 \;|\; (\sum y_i A_i) w = 0 \;\}
=\{(y,w)\mid \,^t y \,^t {B}_i w=0\, (1\leq i \leq 5) \,\}, \\
&U_2=\{ (z,y) \in \mP^4 \times H \;|\; \,^tz (\sum y_iA_i) = 0 \;\}
=\{(z,y)\mid \,^t z {B}_i y=0\, (1\leq i \leq 5) \,\}. \\
\end{aligned}
\]
Corresponding to $Z_1$ and $Z_2$ in (\ref{eqn:diagram0}),
we define 
\[
S_1=\{ w \in \mP^4 \;|\; \det \, ( A_1w \cdots A_5w)=0 \;\},\;
S_2=\{ z \in \mP^4 \;|\; \det \, ( A_1z \cdots A_5z)=0 \;\}. 
\]
As before, we have natural maps $U_i\to H$, $U_i\to S_i$ and 
$\tilde X\to S_i$ $(i=1,2)$ and they are all birational and crepant.

We summarize our construction into a diagram: 
\begin{equation}
\begin{matrix}
\xymatrix{
Y \ar[dd]^{2:1}_{\rho} &  &   &  &  & &  Y\ar[dd]^{2:1}_{\rho} \\
& U_1\ar[dl]_{\pi_2}\ar[dr]^{\pi_1} &\dashleftarrow & \widetilde{X} \ar[dl]\ar[dr]\ar[dd]^{2:1}_{\pi}  &\dashrightarrow & 
U_2\ar[dl]_{\pi_1}\ar[dr]^{\pi_2} & \\
H &  & S_1  &  & S_2 & &  H \\
 &  &   & X &  & &   }
\end{matrix}
\mylabel{eqn:diagram}
\end{equation}
In this diagram, we have included a Calabi-Yau manifold $Y$ which will 
be constructed in the next subsection.
Since $A_i$ are symmetric matrices, obviously $S_1\simeq S_2$ 
and $U_1\simeq U_2$. 
We sometimes abbreviate $U_i$ and $S_i$ as $U$ and $S$ respectively.

We study further properties of the diagram (\ref{eqn:diagram})
assuming only smoothness of $\tilde X$ to the end of this subsection. 

\begin{claim}
\label{claim:smooth}
$U$ is smooth.
\end{claim}
\begin{proof}
$S$ is normal and, moreover, $S$ has only canonical singularities
since $\tilde{X}$ is smooth and $\tilde{X} \to S$ is crepant.
Since $U\to S$ is crepant, $U$ is normal and has only canonical singularities. Let $p\colon U'\to U$ be a crepant terminalization \cite{flop}, i.e.,
$p$ is birational, $U'$ is $\mQ$-factorial and has only terminal singularities, 
and $K_{U'}=p^* K_U$.
Thus $U'$ and $\tilde X$ are two birational minimal models and then
is isomorphic in codimension one.
Since $\rho(\tilde{X})=2$,
it holds that $\rho(U')=2$.
Since $U$ is a complete intersection in $\mP^4\times \mP^4$,
$\rho(U)$ is at least two.
Therefore $U'=U$.
Since two birational minimal models have the same type of singularities 
 \cite{flop},
$U$ is also smooth.  
\end{proof}

Let $L_0$ and $M_0$ be hyperplane sections on $S$ and $H$, respectively, 
and set $L=\pi_1^* L_0, M=\pi_2^* M_0$. 

\begin{proposition}
\label{prop:26}
$U\to H$ is 
a crepant divisorial contraction 
contracting a divisor $E$ to a smooth curve $C$ of degree $20$ 
and genus $26$.  $H$ has an $A_1$ type singularity along $C$.
\end{proposition}

\begin{proof}
Since $\sU\to \sH$ contracts the divisor $\sE$ to 
the locus $\Sigma^*$ of symmetric matrices of rank less than 
or equal to $3$, 
which has codimension two in $\sH$,
$U\to H$ contracts the divisor $E:=\sE|_U$ to the curve $C:=
\Sigma^*|_H$.
Since $U$ is smooth and $\rho(U)=2$,
$U\to H$ is a primitive contraction, hence, $E$ is irreducible 
and, by
 \cite[Theorem 2.2]{Wi}, $C$ is a smooth curve.
Moreover $C$ does not contain the point
corresponding to a rank two quadric.
Indeed, if $C$ contains the point
corresponding to a rank two quadric $Q$,
the fiber of $U\to H$ over $[Q]$ is isomorphic to $\mP^2$
(the vertex of $Q$), and then $E$ is not irreducible, a contradiction.
Hence any fiber of $E\to C$ is $\mP^1$ 
since so is any fiber of $\sE\to \Sigma^*$
over the point corresponding to a rank three quadric.
Therefore $H$ has an $A_1$ type singularity along $C$.
Since $\deg \Sigma^*=20$  by the formula \cite[Proposition 12]{HTu} 
and $C$ is a linear section of $\Sigma^*$,
$C$ has the same degree.
To show $g(C)=26$, we use the identity
$2L=4M-E$, which follows from
Proposition \ref{prop:relation}.
We have the following table of intersections:
\begin{itemize}
\item $M^3=5$ and $L^3=5$ since $H$ and $S$ are quintics.
\item
$M^2 E=0$ since $M^2 E=(M_{|E})^2$ and $M_{|E}$ is numerically
a sums of fibers of $E\to C$.
\item
$ME^2=-2\deg C=-40$.
Indeed, 
note that $ME^2=M_{|E} E_{|E}$.
Since $M_{|E}$ is the pull-back of a hyperplane section of $C$ and
$E_{|E}\cdot f=-2$ for a fiber $f$ of $E\to C$,
we have $ME^2=-2\deg C$.
\item
$E^3=(K_U+E)^2 E=(K_E)^2=8(1-g(C))$
since $K_U=0$.
\end{itemize}
Therefore we have 
\[
40=(2L)^3=(4M-E)^3=
64\times 5-12\times 40 -E^3=-160-8(1-g(C)).
\]
Hence $g(C)=26$.
\end{proof}

\begin{claim}
\mylabel{claim:casediv}
The type of $\tilde X\to S$ is one of the following:
\begin{enumerate}[(1)]
\item
$\tilde X \to S$ contracts an irreducible divisor to a curve, or
\item
$\tilde X \to S$ contracts only a finite number of rational curves.
\end{enumerate}
\end{claim}

\begin{proof}
Since $\rho(\tilde X)=2$ and $\tilde X$ is smooth,
$\tilde X\to S$ satisfies (1) or (2),
or $\tilde{X}\to S$ contracts an irreducible divisor to a point.
We exclude the last possibility.
Suppose $\tilde{X}\to S$ contracts an irreducible divisor to a point.
We denote by $E_i$ the exceptional divisor of $\tilde X\to S_i$
($i=1,2$).
In this case, $S$ is $\mQ$-factorial, thus the image of $E_1$ on $S_2$
is an ample divisor since $\rho(S_2)=1$.
This implies that $E_1\cap E_2\not =\emptyset$ since $E_1$
is not ample on $\tilde X$.
Then, however, the curve $E_1\cap E_2$ is contracted by both $\tilde X\to S_1$
and $\tilde X\to S_2$, a contradiction since 
$\tilde X\to S_1$
and $\tilde X\to S_2$ are distinct and primitive.
\end{proof}

Assume that $\tilde X \rightarrow S$ contracts an irreducible divisor to 
a curve. We can construct examples for this situation. 
For example, the linear system $P$ 
determined by the following symmetric matrices $A_i$ is regular and 
has the property:
\[
\sum_{k=1}^5 y_k A_k = 
\left( \begin{smallmatrix}
  y_2 &  y_1 &  0 & 0  & y_5 \\
  y_1 &  y_3 & y_2 & 0  &  0 \\
  0  &  y_2 & y_4 & y_3 &  0 \\
  0  &  0  & y_3 & y_5 & y_4 \\
  y_5 &  0  &  0 & y_4 & y_1  \end{smallmatrix} \right).
\]
In this case,
$S$ is $\mQ$-factorial and the image of the divisor is $\Sing S$.
Hence $U\to S$ cannot be a small contraction and then also contracts a divisor.
The image of the exceptional divisor of $U\to S$ is also $\Sing S$. 
By \cite[Theorem 2.2]{Wi} and its proof,
both $\tilde X\to S$ and $U\to S$ are the blow-ups 
along $\Sing S$, thus they are isomorphic over $S$.
If we identifies $U_1$, $\tilde X$, and $U_2$,
then $U_1\to H$ is identified with $\tilde X\to S_2$.
Therefore we can simplify the diagram (\ref{eqn:diagram}) as follows:
\begin{equation}
\mylabel{eqn:diagram1}
\begin{matrix}
\xymatrix{
Y\ar[dd]^{2:1}_{\rho} &  &  Y\ar[dd]^{2:1}_{\rho} \\
& \widetilde{X} \ar[dl]\ar[dr]\ar[dd]^{2:1}_{\pi} & \\
H &  & H \\
&  X .}
\end{matrix}
\end{equation}

Assume that $\tilde X\to S$ contracts only finite number of rational curves.
This is a general situation; if $P$ is general, then we can verify 
by computer calculations that 
$\tilde{X}\to S$ contracts $50$ disjoint $\mP^1$'s 
to 50 ordinary double points of $S$.
In this case, $U\to S$ is also a small contraction since $S$ is not
$\mQ$-factorial. Moreover, $U$ and $\tilde X$ are not isomorphic
since both of the Picard numbers are two and the types of 
contractions they have
are different; $U$ has a divisorial and a small contraction
but $\tilde X$ has two small contractions.

\noindent
{\bf (3-5)} {\it Construction of a double covering $Y$}: 
Set $N=2M-L$ and $N_0:={\pi_2}_*N$. By Proposition \ref{prop:relation},
we have $2N\sim E$ whence $2N_0\sim 0$.
\begin{proposition}
$N_0 \not \sim 0$.
\end{proposition}

\begin{proof}
By definition $M_0={\pi_2}_*M$.
Assume by contradiction that $N_0 \sim 0$, equivalently, $\pi_{2*} 
L \sim 2 M_0$.
Then $\pi_{2*}L$ is a Cartier divisor,
hence we may write $L=\pi_2^*(\pi_{2*} L)-b E$ for some integer $b$.
Therefore we have $N=\pi_2^* N_0+bE\equiv bE$.
Since $N\equiv \frac{1}{2} E$, we have 
$bE\equiv \frac{1}{2} E$, a contradiction.
\end{proof}

Now we can take the double cover of $H$ associated to the $2$-torsion 
Weil divisor $N_0$, namely, 
\begin{equation}
Y:=\sSpec_{H} \,(\sO_H\oplus \sO_H(N_0) )\;,
\mylabel{eqn:defY}
\end{equation}
where $\sO_H\oplus \sO_H(N_0)$ has a ring structure by using a 
nowhere vanishing section of $2 N_0\sim 0$. The natural projection 
$\rho\colon Y\to H$ is ramified along the curve $C$ and 
is \'etale outside $C$. 
Since it is analytically locally a universal cover along $C$, we see that 
$Y$ is smooth. Set $\widetilde{M}=\rho^* M_0$.

\begin{proposition} \mylabel{prop:Y}
The $3$-fold $Y$ is a Calabi-Yau $3$-fold with $\widetilde{M}^3=10$.
\end{proposition}   
 
\begin{proof}
By construction, we have $K_Y\sim 0$ and $\widetilde{M}^3=10$.
We verify $h^i(\sO_Y)=0$ for $i=1,2$.
Indeed, by $\rho_* \sO_Y=\sO_H\oplus \sO_H(N_0)$,
we have $h^i(\sO_Y)=h^i(\sO_H)+h^i(\sO_H(N_0))=h^i(\sO_H(N_0))$.
Thus it suffices to show $h^i(\sO_H(N_0))=0$.
Note that a similar construction to the above works for a general $5$-dimensional linear system $\overline{P}$ of quadrics in $\mP^4$. We choose such a $\overline{P}$ containing $P$.
We obtain several objects corresponding to those for $P$.
We denote them by putting overlines 
to the corresponding objects for $P$.
First of all, note that $\overline{H}$ is a quintic hypersurface in $\mP^5$,
thus it is a Fano $4$-fold. For a general $\overline{P}$,
$\overline{H}$ does not contain the point corresponding to
a rank two quadric and then $\overline{H}$ has only ordinary double points
as its singularities.
Moreover, $\overline{S}=\mP^4$.
By Proposition \ref{prop:relation},
we can show that $2\overline{L}=4\overline{M}-\overline{E}$.
We set $\overline{N}_0:=\overline{\pi_2}_*(2\overline{M}-\overline{L})$.
Then $\overline{N}_0$ is a $2$-torsion Weil divisor on $\overline{H}$ and
$\overline{N}_0|_{H}=N_0$.
Consider the exact sequence
\[
0\to \sO_{\overline{H}}(\overline{N}_0-H)\to \sO_{\overline{H}}(\overline{N}_0)
\to \sO_H (N_0)\to 0.
\]
Since $\overline{H}$ is Fano,
we have 
\[
h^i(\sO_{\overline{H}}(\overline{N}_0))=
h^i(\sO_{\overline{H}}(\overline{N}_0-K_{\overline{H}}+K_{\overline{H}}))=0\,
(i=1,2)
\]
by the Kawamata-Viehweg vanishing theorem.
Similarly,
\[
h^{i+1}(\sO_{\overline{H}}(\overline{N}_0-H))=
h^{3-i}(\sO_{\overline{H}}(-\overline{N}_0+H+K_{\overline{H}}))=0
\, (i=1,2).
\]
Therefore we have $h^i(\sO_H(N_0))=0$ $(i=1,2)$.
\end{proof}

\begin{proposition} \mylabel{prop:Y-c2j}
$c_2(Y)\cdot \widetilde{M}=40$, $e(Y)=-50$ and $\rho(Y)=1$.
\end{proposition}

\begin{proof}
First let us compute $c_2(Y)\cdot \widetilde{M}$.
Take a general member $D \in |M_0|$ and set $\widetilde{D}:=\rho^*D$
and $\rho'=\rho|_{\widetilde{D}}$.
By the Bertini theorem, $\widetilde{D}$ is smooth,
and $D$ intersects $C$ transversely at $20$ points,
and $D$ has ordinary double points at $D\cap C$.
By the standard exact sequence
\[
0\to T_{\widetilde{D}}\to {T_Y}|_{\widetilde{D}}\to 
\sO_{\widetilde{D}}(\widetilde{D})\to 0,
\]
we have
\[
c_2(Y)\cdot \widetilde{M}=c_2({T_{Y}}|_{\widetilde{D}})=
c_1(\widetilde{D})\cdot \widetilde{D}|_{\widetilde{D}}+c_2(\widetilde{D})
=c_2(\widetilde{D})-10.
\]
By the Noether formula,
we have $c_2(\widetilde{D})=12\chi(\sO_{\widetilde{D}})-c_1^2(\widetilde{D})
=12\chi(\sO_{\widetilde{D}})-10$.
Since $\rho'$ is the restriction of the double cover 
$\rho$,
we have
$\rho'_*\sO_{\widetilde{D}}=\sO_{D}\oplus \sO_D(n_0)$,
where $n_0=N_0|_D$.
Thus $\chi(\sO_{\widetilde{D}})=\chi(\sO_{{D}})+\chi(\sO_{{D}}(n_0))$.
Since $D$ is a quintic surface, it holds
$\chi(\sO_{{D}})=h^0(\sO_{{D}})+h^0(K_{{D}})=5$. 
To compute $\chi(\sO_{{D}}(n_0))$, 
we use the singular Riemann-Roch theorem for surfaces with only Du Val 
singularities \cite[Theorem 9.1]{Re}.
Then, noting that $2n_0\sim 0$ and $n_0$ is not Cartier at $20$ 
ordinary double points of $D$, we have 
$\chi(\sO_{{D}}(n_0))=\chi(\sO_{{D}})-\frac{1}{4}\times 20=0$. 
Consequently we have  
$c_2(Y)\cdot \widetilde{M}=\chi(\sO_{{D}})-20=40$.

Second we compute $e(Y)=-50$.
Recalling $\widetilde{X}\simeq U$ or $\widetilde{X}$ and $U$ are 
connected by a flop,
we have $e(U)=e(\widetilde{X})=-100$ (cf. \cite{flop}).
Since $2N=E$, we can take the double covering $\widetilde{U}\to U$
branched along $E$. 
By $2e(U)=e(\widetilde{U})+e(E)$, we have $e(\widetilde{U})=-100$.
This double covering is compatible with $Y\to H$
as follows:
\begin{equation*}
\xymatrix{\widetilde{U}\ar[d]\ar[r]^{\widetilde{\pi_2}} & Y\ar[d]^{\rho} \\
U\ar[r]^{\pi_2} & H,
}
\end{equation*}
where $\widetilde{\pi_2}$ is the blow-up along the smooth curve 
$\rho^{-1}{C}\simeq C$. 
Since $e(C)=-50$, we have
$e(Y)=-50$.

Finally we show $\rho(Y)=1$.
Since $e(Y)=2(h^{1,1}(Y)-h^{1,2}(Y))=-50$,
we have only to show $h^{1,2}(Y)=26$.
It is well-known that $h^{1,2}(Y)$ is the number of moduli of $Y$.
First we compute the number of moduli of $H$ 
by the same way as \cite[Proposition 2.26]{Ca} as follows.
If $H$ is isomorphic to another symmetric determinantal quintic $H'$ in $P$,
then this isomorphism is induced by a projective automorphism of $P$
since $\rho(H)=\rho(H')=1$ and hence the isomorphism
preserve the primitive polarization. 
An automorphism $H\simeq H$ induces an automorphism $Y\simeq Y$.
Since $Y$ is a smooth Calabi-Yau $3$-fold,
its automorphism group is finite, thus so is the automorphism group of $H$.
Therefore the number of moduli of $H$ is 
$\dim \mP(\Hom (\mC^5, \mathrm{Sym}^2 \mC^5))/\Aut P\times \Aut \mP^4
= 26$.
Next note that any deformation of $Y$ comes from that of the pair
$(Y,\tilde{M})$ since $h^2(\sO_Y)=0$ by \cite[p.151, Proposition 3.3.12]{Se}.
Hence a deformation of $Y$ induces that of $H$,
and then the number of moduli of $Y$ is less than or 
equal to that of $H$. Therefore $h^{1,2}(Y)\leq 26$, equivalently,
$h^{1,1}(Y)\leq 1$, which must be an equality.
\end{proof}

\begin{remark}
The surface $D$ and $\widetilde{D}$ constructed in the proof of
Proposition \ref{prop:Y-c2j} is originally studied deeply by
F.~Catanese \cite{Ca} (see also \cite{Be}).
Our $Y$ and $H$ are three-dimensional counterpart of his surface. 
\end{remark}

From Proposition \ref{prop:Y} and Proposition \ref{prop:Y-c2j}, we 
conclude that

\begin{theorem} \mylabel{thm:main-thm}
The double cover $Y=Spec_H ({\mathcal O}_H \oplus 
{\mathcal O}_H(N_0))$ of the Hessian quintic is the Calabi-Yau threefold 
predicted in Conjecture 1. 
\end{theorem}

\vskip1cm

\section{{\bf Counting curves and discussions}}

\noindent
{\bf (4-1)} {\it  Counting curves on $X$ and $Y$}: We verify some integer 
numbers called BPS numbers which we read from  
the higher genus Gromov-Witten invariants determined by mirror 
symmetry (see Appendix A). 
We observe that some numbers have good interpretations from the geometry 
of Calabi-Yau manifolds $X$ and $Y$.  

\vskip0.2cm
\noindent
(4-1.1) {\it Genus 0 curve of degree 1 on $X$, $n_0^X(1)=50$}:  
If $\tilde{X}\not =U$, then
there are exactly $50$ exceptional curves ${l_j}^i$  
of $\tilde{X}\to S_i$ in general ($i=1,2$, $j=1,\dots, 50$).
We label them so that ${l_j}^1$ and ${l_j}^2$ are exchanged by 
the involution. Since $\deg {l_j}^i=1$ with respect to the $(1,1)$ 
divisor class on $\tilde{X}$, the image $n_j$ of ${l_j}^1$ and 
${l_j}^2$ on $X$ have degree $1$.
Thus $n_j$ are lines on $X$ and are mutually distinct.

If $\tilde{X}=U$, then $S_i=H$ and non-trivial fibers of $\tilde{X}\to S_i$
are parametrized by $C$, whose Euler number is $-50$ 
(see Proposition \ref{prop:26}).
Similarly to the generic case, we can show that
there is a family of lines on $X$ parametrized by $C$. According to 
the Gromov-Witten theory, we count the BPS number by 
$n_1^X(1)=(-1)^{\dim C} \, e(C)=50$.

\vskip0.3cm
\noindent
(4-1.2) {\it Genus 1 curve of degree 2 on $Y$, $n_1^Y(2)=50$}: 
If $\tilde{X}\not =U$, then
there are exactly $50$ exceptional curves $\ell_j$  
of $U\to S$ in general ($j=1,\dots, 50$).
Note that $\ell_j$ are the strict transforms of lines $m_j$ on $H$, thus
$M\cdot \ell_j=1$.
By the identity $2L=4M-E$ and $L\cdot \ell_j=0$,
$\ell_j$ intersects $E$ at four points counted with multiplicities.
\[
\text{(*) {\it 
For a general $P$, any $l_j$ intersects $E$ at four points transversally.}}
\hskip1.3cm \;
\]

\begin{proof}
This follows from a simple dimension count.
We make use of the diagram
$ \mP^4 \xleftarrow{\pi_1}\sU \xrightarrow{\pi_2} \sH$ 
defined in the subsection (3-3). Recall the notation there.
We count the dimension of the locus $\sS_{\sE}$
of lines in fibers which are tangent to $\sE$.
A fiber $F$ of $\sU\to \mP^4$ is isomorphic to $\mP^9$
and $\sE|_F$ is a quartic hypersurface.
Therefore, in $G(2,F)$, the locus of tangent lines
to $\sE|_F$ is $15$-dimensional.
Thus $\dim \sS_{\sE}=15+4=19$.
Let $l$ be a line in $(\mP^{14})^*$. Then, in $G(4, (\mP^{14})^*)$, 
the locus of $4$-planes containing $l$ is $30$-dimensional.
Therefore, in $G(4, (\mP^{14})^*)$, the locus 
of $4$-planes containing 
the image of at least one tangent line to $\sE$
is $(30+19)$-dimensional. On the other hand, 
$\dim G(4, (\mP^{14})^*)=50$.
Hence a general $P\in G(4, (\mP^{14})^*)$ does not contain
the images of tangent lines to $\sE$.
\end{proof}
Due to the property (*), 
$m_j$ intersects with $C$ at four points transversally. 
Therefore by taking the 
double cover $Y\to H$, the inverse image of $m_j$ on $Y$ is an 
elliptic curve of degree $2$.

If $\tilde{X}=U$, then 
we can similarly show that
there is a family of elliptic curves on $Y$ of degree $2$
parametrized by $C$. According Gromov-Witten theory again, 
we count the Euler number of $C$ for the BPS number 
$n_1^Y(2)=(-1)^{\dim C} \, e(C)=50$.  

\vskip0.3cm
\noindent
(4-1.3) {\it Genus $7$ curve of degree $8$ on $Y$, $n^Y_{7}(8)=150$}: 
We show that for each elliptic curve of degree two as in (4-1.2),
there exists a family of genus $7$ curves of degree $8$ parametrized
by $\mP^2$. Indeed, an elliptic curve $\widetilde{\delta}$ as in (4-1.2)
is the inverse image of a $4$-secant line $\delta$ of $C$.
Let $\Pi$ be a plane in $P$ containing $\delta$.
Note that such planes $\Pi$ are parametrized by a copy of $\mP^2$.
Then $\Pi\cap H$ is the union of $\delta$ and 
a curve $\gamma$ of degree $4$ and (arithmetic) genus $3$.
For a generic $\Pi$, $\gamma$ is smooth, and $\gamma$
intersects 
the ramification curve $C$ of $\rho$ at the four points $C\cap \delta$.
Moreover,
the ramification $\widetilde{\gamma}:=\rho^{-1}(\gamma)\to \gamma$
occurs only at these four points and is simple.
Hence $\deg \widetilde{\gamma}=8$ and,
we can compute the genus of $\tilde \gamma$: 
$2g(\widetilde{\gamma})-2=2(2g(\gamma)-2)+4$, i.e., $g(\widetilde{\gamma})=7$.

Elliptic curves as in (4-1.2) are parametrized by
50 points or the curve $C$ with $e(C)=-50$.
Therefore 
$n^Y_{7}(8)=50 \times e(\mP^2)=150$.

\vskip0.3cm
\noindent
(4-1.4) {\it Genus 6 curve of degree 10 on $X$, $n_6^X(10)=5$}:  
We construct a family of curves of genus $6$ and degree $10$
which are parametrized by $(\mP^4)^*$. Then we can explain the BPS number 
$n_6^X(10)=5$ by $(-1)^{\dim (\mP^4)^*} \, e((\mP^4)^*)=5$.
Let $L \cong \mP^3$ be any hyperplane in $\mP^4$ and
set $C_L:={\rm Sym}^2 L\cap X$.
$C_L$ is a linear section of ${\rm Sym}^2 L$ 
since $X$ is a linear section of ${\rm Sym}^2 \mP^4$.
Since the degree of $L\times L$ is $20$, the degree of 
${\rm Sym}^2 L$ is $10$.
We see that $C_L$ is a curve. Indeed, otherwise $C_L$ contains 
a $2$-dimensional component and
its degree is less than or equal to $10$
because $C_L$ is a linear section of ${\rm Sym}^2 L$.
However, $\Pic X$ is generated by its hyperplane section,
whose degree is $35$, a contradiction.
Thus $C_L$ is a curve, and then
$C_L$ is a linear complete intersection in ${\rm Sym}^2 L$.
Note that $K_{{\rm Sym}^2 L}=-4D$, where $D$ is a hyperplane section of 
${\rm Sym}^2 L$. 
Thus $K_{C_L}=D|_{C_L}$, and then $\deg K_{C_L}=(D^6)_{{\rm Sym}^2 L}=10$.
This means that the arithmetic genus of $C_L$ is $6$. We can prove 
that the curve $C_L$ is smooth for a generic $L$. To prove this, 
set $\widetilde{C}_L:=(L\times L)\cap \widetilde{X}$.
Since the morphism $\widetilde{C}_L\to C_L$ is the quotient
by the fixed point free involution,
it suffices to show that $\widetilde{C}_L$ is smooth for a generic $L$.
Note that $L\times L\subset \mP^4\times \mP^4$ is 
the scheme of zeros of a section
of the vector bundle $\sO_{\mP^4\times \mP^4}(1,0)
\oplus \sO_{\mP^4\times \mP^4}(0,1)$.
Therefore $\widetilde{C}_L$ is 
the scheme of zeros of a section
of the vector bundle $\sE_{\widetilde{X}}:=
\sO_{\mP^4\times \mP^4}(1,0)|_{\widetilde{X}}
\oplus \sO_{\mP^4\times \mP^4}(0,1)|_{\widetilde{X}}$.
Moreover we may choose a symmetric section defining $\widetilde{C}_L$.
Since $\sE_{\widetilde{X}}$ is generated by symmetric sections,
a generic $\widetilde{C}_L$ is smooth by the Bertini theorem for
vector bundles \cite{Muk3}. 

Here we extract the relevant invariants from the Table 2 in Appendix A: 
\begin{equation}
\begin{array}{c|ccccc}
g & 0 & \cdots & 4 & 5 & 6 \\
\hline
n_g^X(10) & 80360393750 & \cdots & 1713450 & 100 & 5 \\
\end{array}
\mylabel{eqn:tableX}
\end{equation}
It is worth while remarking that we see similar curves 
in the Calabi-Yau threefolds given by linear sections of $Gr(2,7)$, 
$X_G:=Gr(2,7)_{1^7}$, whose BPS numbers may be found in \cite{HK} up to 
$g \leq 5$\footnote{The data up to $g \leq 10$ is available 
at \hbox{http://www.ms.u-tokyo.ac.jp/\~\,hosono/GW/GrPf\_html}}. 
From the table we read
the relevant part:
\begin{equation}
\begin{array}{c|ccccc}
g & 0 & \cdots & 6 & 7 & 8 \\
\hline
n_g^{X_G}(14) & 26782042513523921505 & \cdots & 123676 & 392 & 7 \\
\end{array}
\mylabel{eqn:tableGr}
\end{equation}
The BPS number $n_8^{X_G}(14)$ can be explained by a family of genus 8 
and degree 
14 curves parametrized by $(\mP^6)^*$. This family follows from hyperplanes 
$L\cong \mP^5$ in $\mP^6$ as above and $C_L = \wedge^2 L \cap X_G$. For 
generic $L$, $C_L$ is a curve given by a transverse linear section of 
$G(2,6)$, which has genus $8$, degree $14$, and was studied in details in 
 \cite{Muk2}. Again, we explain $n_8^{X_G}(14)=7$ by  
$(-1)^{\dim (\mP^6)^*} e((\mP^6)^*)$. 

\vskip0.3cm
\noindent
(4-1.5) {\it BPS numbers from nodal curves}: 
The BPS numbers with a fixed degree $d$ are indexed by geometric 
genus $h$ which ranges from zero to the arithmetic genus $g=g_d$, 
as typically shown in (\ref{eqn:tableX}) and (\ref{eqn:tableGr}). 
This structure was introduced based 
on the intuitions from physics in \cite{GV}, and possible geometric 
explanations of this have been given in \cite{KKV}. According to the latter 
formulation, the number $n_h(d)$ counts the Euler number (up to sign) of 
the locus on the moduli space where the curves have $g-h$ nodes. 
Based on this, it has been proposed for the BPS numbers of genus $g-1$ 
that 
\begin{equation}
n_{g-1}(d) = -(-1)^{\dim {\mathcal M}_d^{g_d}} \{ e({\mathcal C}) + (2g-2) 
e({\mathcal M}_d^{g_d}) \},
\mylabel{eqn:katz}
\end{equation}
where ${\mathcal M}_d^{g_d}$ is the parameter space of the curves of 
degree $d$ and the arithmetic genus $g=g_d$, and ${\mathcal C}$ is 
the universal curves over ${\mathcal M}_d^{g_d}$. Some of BPS numbers in 
our tables can be reproduced 
easily by (\ref{eqn:katz}). 

\vskip0.2cm
(i) {\it Genus 5 curve of degree 10 on $X$, $n_5^X(10)=100$}:   
In this case, the universal curve ${\mathcal C}$ has a natural 
fibration  over $X$ with fiber $(\mP^2)^*$. To see this fibration, we note 
that an element of the universal curve is given by 
$C_L = {\rm Sym}^2 L \cap X$.  We fix a point $(z,w)$ on $X$. Then, since 
$X$ is smooth, $z\not=w$ and the condition $(z,w) \in {\rm Sym}^2 L$ 
imposes two linearly independent conditions on the choice of $L$ 
parametrized by $(\mP^4)^*$. From this, we obtain the claimed fibration, 
and $e({\mathcal C})=3\times e(X) = -150$. Then we 
can evaluate the formula (\ref{eqn:katz}) as $n_5(10)=-( -150 + 10 \times 
5)=100$ verifying the BPS number $n_5^X(10)$ in (\ref{eqn:tableX}). 

The BPS number $n_7^{X_G}(14)$ in (\ref{eqn:tableGr}) can be verified 
exactly in the same way by applying the formula (\ref{eqn:katz}). 
In this case, an element of the universal curve is given by 
$C_L=\wedge^2 L \cap X_G$ parametrized by $(\mP^6)^*$.  
Fixing a point $\xi \in X_G$ determines a line in $\mP^6$ and entails two 
linearly independent conditions on the choice of $L$, and hence gives rise 
to a fibration over $X_G$ with fiber $(\mP^4)^*$. We evaluate the formula 
(\ref{eqn:katz}) by $n_6(14)=-(5\times e(X_G)+14 \times 7)$ with 
$e(X_G)=-98$ verifying the BPS number $n_6^{X_G}(14)$ in (\ref{eqn:tableGr}).

\vskip0.2cm
(ii) {\it Genus $11$ curve of degree $10$ on $Y$, $n^Y_{11}(10)=10$}: 
We construct a family of curves of genus $11$ and degree $10$
which are parametrized by $G(2,P)$. Then we can explain the BPS number 
$n_{11}^Y(10)=10$ by $(-1)^{\dim G(2,P)} \, e(G(2,P))=10$.
Let $\Pi\subset P$ be a plane in $P$. Then $C_{\Pi}:=\Pi\cap H$ is 
a plane curve of degree $5$ and arithmetic genus $6$. 
Let $\widetilde{C}_{\Pi}:=\rho^{-1} (C_{\Pi})$.
We show that $\widetilde{C}_{\Pi}$ is a genus $11$ curve of degree $10$
for a generic $\Pi$ (in general, the arithmetic genus of $\widetilde{C}_{\Pi}$
is $11$). Indeed, for a generic $\Pi$, $C_{\Pi}$ is smooth
and is disjoint from the ramification curve $C$ of $\rho$.
Therefore $\widetilde{C}_{\Pi}\to C_{\Pi}$ is \'etale and its degree is $10$.
By the Hurwitz formula, we can compute the genus of $\widetilde{C}_{\Pi}$:
$2g(\widetilde{C}_{\Pi})-2=2(2g({C}_{\Pi})-2)=20$, i.e., 
$g(\widetilde{C}_{\Pi})=11$.

\vskip0.2cm
(iii) {\it Genus $10$ curve of degree $10$ on $Y$, $n^Y_{10}(10)=100$}: 
In a similar way to (4-1.5)(i), we interpret this number 
from $n^Y_{11}(10)=10$.
The universal curve $\sC$ over $\sM_{10}^{11}$
has a natural fibration over $Y$ with fiber $G(2,4)$.
Indeed, for a point $y\in Y$, planes in $P$ through $\rho(y)$ is parametrized
by a copy of $G(2,4)$. Thus we see that $e(\sC)=e(G(2,4))e(Y)=-300$.
By (4.3), we have $n_{10}^Y(10)=100$.

\vskip0.5cm
\noindent
{\bf (4-2)} {\it Discussions}: The similarity of the curves $C_L$ in their 
construction above is not accidental. In fact, we have  
$X={\rm Sym}^2 \mP^4 
\cap H_1\cap \cdots \cap H_5 \subset \mP({\rm Sym}\mC^5)$ (see (3-4)), while 
for the linear sections of the Grassmannian 
$X_G = Gr(2,7) \cap H_1 \cap \cdots \cap H_7 \subset \mP(\wedge^2 \mC^7)$. 
The Mukai dual $X_G^\sharp$ of $X_G$ is nothing but the 
Pfaffian Calabi-Yau threefold, and the derived equivalence between 
$X_G$ and $X_G^\sharp$ follows from the projective homological duality 
 \cite{Ku1} or from the explicit construction of the kernel of the 
equivalence \cite{BCa}. Our  Calabi-Yau threefold $Y$ has been defined 
as the branched covering of the Hessian quintic which we have called 
a shifted Mukai dual in (3-4).  
The necessity of the shifting is a well-known 
phenomenon in the projective homological duality \cite{Ku2}. 
However the covering construction of $Y$ is a new ingredient that appeared 
in the projective geometry of the Reye congruence $X$. 
A general framework 
to incorporate this covering as well as the relevant projective duality 
seems to be required [HT]. 

\newpage
\begin{appendix}

\section{ BPS numbers of $X$ and $Y$}

For a Calabi-Yau manifold $\bullet$, the BPS 
numbers $\{ n_g^\bullet(d) \}$ are read off from the Gromov-Witten invariants 
$\{N_g^\bullet(d)\}$ through the following formula proposed from 
the arguments in physics \cite{GV}:  
\begin{equation}
\sum_{g \geq 0} N_g^\bullet(d) \lambda^{2g-2} = 
\sum_{k|d} \sum_{g\geq0} n_g^\bullet(d/k) \frac{1}{k} 
\big( 2 \sin \frac{k\lambda}{2} \big)^{2g-2} \;\;. 
\mylabel{eqn:GW-BPS}
\end{equation}

{\tiny
\[
 \begin{array}{|c | l l l l|}
\hline
d & g=0 & g=1 & g=2 & g=3 \cr
\hline
   1 & 50 & 0 & 0 & 0 \\
    2 & 325 & 0 & 0 & 0 \\
    3 & 1475 & 275 & 0 & 0 \\
    4 & 15325 & 4400 & 0 & 0 \\
    5 & 148575 & 84866 & 0 & 0 \\
    6 & 1885575 & 1583175 & 4400 & 0 \\
    7 & 24310650 & 30888200 & 536200 & -550 \\
    8 & 348616525 & 604676675 & 29838375 & 14350 \\
    9 & 5158310775 & 12044071475 & 1207458375 & 11555950 \\
    10 & 80360393750 & 241743300988 & 42169242200 & 1184416575 \\
    11 & 1287049795175 & 4897366348600 & 1341159720325 & 78638706125 \\
    12 & 21247935013725 & 99853114108900 & 40099463511075 & 4132376105575 \\
    13 & 358438400398475 & 2048292673120975 & 1146369224007075 &
      187040727391700 \\
    14 & 6171544153689825 & 42223845013663600 & 31689082612611600 &
      7622009409916000 \\
    15 & 108035835968890075 & 874235620542355546 & 853281016802276675 &
      287256247725860175 \\
\hline
\hline
d & g=4 & g=5 & g=6 & g=7 \cr
\hline
   1 & 0 & 0 & 0 & 0 \\
   : & : & : & : & : \\
    8 & 0 & 0 & 0 & 0 \\
    9 & -100 & 0 & 0 & 0 \\
    10 & 1713450 & 100 & 5 & 0 \\
    11 & 722228850 & 579975 & -200 & 0 \\
    12 & 103406622475 & 399883675 & 2000 & 0 \\
    13 & 9444029474450 & 119293055775 & 143795975 & 4425 \\
    14 & 662595383138450 & 18997158932125 & 121118688975 & 18994975 \\
    15 & 38944123712509700 & 2069810853535000 & 35619476408010 & 114318975250
      \\
    16 & 2013111720646080925 & 175344589016056100 & 6103559347336200 &
      65977694328550 \\
    17 & 94357347971238632650 & 12406278936653253300 & 745756863552486250 &
      17913359215581675 \\
    18 & 4093411060955729478975 & 766225906495401998275 & 71992613535494881825
      & 3146508705290392100 \\
\hline
\hline
d & g=8 & g=9 & g=10 & g=11 \cr
\hline
  1 & 0 & 0 & 0 & 0 \\
    : & : & : & : & : \\
    13 & 0 & 0 & 0 & 0 \\
    14 & 57450 & 0 & 0 & 0 \\
    15 & -411925 & 150 & 0 & 0 \\
    16 & 105464925850 & -14495125 & 0 & 0 \\
    17 & 126509159387650 & 97409049625 & -16452900 & 4125 \\
    18 & 54474919888610025 & 262131821055725 & 101459279850 & -16598650 \\
    19 & 13654329297710632350 & 177453757177942875 & 613991099856975 &
      132790770150 \\
    20 & 2403664871323486315025 & 62687258689222376200 & 637194272508657750 &
      1682171918115410 \\
\hline
\end{array} \]
\[
 \begin{array}{|c | l l l|}
\hline
d & g=12 & g=13 & g=14 \cr
\hline
  1 & 0 & 0 & 0 \\
    : & : & : & : \\
    17 & 0 & 0 & 0 \\
    18 & 4125 & 0 & 0 \\
    19 & -14185500 & 0 & 0 \\
    20 & 261281853600 & 8586100 & 500 \\
    21 & 5515988462714275 & 847491960225 & 111934650 \\
    22 & 11952342577834431775 & 21928222551486975 & 3692486074300 \\
    23 & 10414582838049361140700 & 64120595692604535100 & 105978121788664050
      \\
    24 & 5202734853892057059278550 & 71682320156336419778925 &
      400537217143155608900 \\
    25 & 1763658604119550663892185975 \quad\, 
       & 44449075337829652624588375 \quad\, &
      559281092909472100209725  \quad \\
\hline
\end{array} \] }

\centerline{Table 2. BPS numbers $n_g^X(d)$ of the Reye congruence $X$ 
up to $g=14$.}

\newpage

{\tiny 
\[
 \begin{array}{|c | l l l |}
\hline
d & g=0 & g=1 & g=2 \cr
\hline
    1 & 550 & 0 & 0 \\
    2 & 19150 & 50 & 0 \\
    3 & 1165550 & 8800 & 0 \\
    4 & 106612400 & 2205050 & 1475 \\
    5 & 12279982850 & 571891188 & 3421300 \\
    6 & 1623505897500 & 145348456975 & 2866211100 \\
    7 & 235773848446900 & 36407863802400 & 1629661401800 \\
    8 & 36677428272627500 & 9007535459995025 & 754167588991150 \\
    9 & 6013303354178962900 & 2206445978587108100 & 306780142847805350 \\
    10 & 1027574672771836629150 & 536379778481962721435 &
      114213423234737792750 \\
    11 & 181575714674126153579750 & 129650593880351912045400 &
      39864084273176413792000 \\
    12 & 32984176212221309289675600 & 31205903422813957619133225 &
      13248719910964644190631550 \\
\hline
\hline
d & g=3 & g=4  & g=5 \cr
\hline
 1 & 0 & 0 & 0 \\
    2 & 0 & 0 & 0 \\
    3 & 0 & 0 & 0 \\
    4 & 0 & 0 & 0 \\
    5 & 100 & 0 & 0 \\
    6 & 8187125 & 825 & 0 \\
    7 & 19888077100 & 37966000 & 8600 \\
    8 & 23146244434500 & 212732689425 & 320681675 \\
    9 & 18532051742762800 & 454873290963350 & 3593733857100 \\
    10 & 11781287359871427050 & 585882543518117475 & 12728758671277375 \\
    11 & 6399895657742481953000 & 551632827711119126550 & 24538387436834227350
      \\
    12 & 3101222971082783163969650 & 419397397316501722086175 &
      32465476031958878167700 \\
    13 & 1377826723507606278941823100 & 272882663054775849898040250 &
      33183767012793770336079250 \\
\hline
\hline
d & g=6 & g=7  & g=8 \cr
\hline
   1 & 0 & 0 & 0 \\
   : & : & : & : \\
    6 & 0 & 0 & 0 \\
    7 & 0 & 0 & 0 \\
    8 & -10750 & 150 & 0 \\
    9 & 4127030600 & -2187050 & 1650 \\
    10 & 94413759088805 & 85277244900 & -143177275 \\
    11 & 506498658506640250 & 3797470967882100 & 3585047803950 \\
    12 & 1368477183702651373550 & 28303440043814561275 & 229236705583292025 \\
    13 & 2423846922015534436170200 & 101000317775360734233450 &
      2180564163043456331350 \\
    14 & 3208929556209610071440200300 & 229260894125697202286776775 &
      9748272270449563477480275 \\
\hline
\hline
d & g=9 & g=10  & g=11 \cr
\hline
    1 & 0 & 0 & 0 \\
    : & : & : & : \\
    8 & 0 & 0 & 0 \\
    9 & 0 & 0 & 0 \\
    10 & 6875 & 100 & 10 \\
    11 & -7611379550 & 3258450 & -6600 \\
    12 & 319959943617050 & -378820282275 & 942134550 \\
    13 & 20098911841165263850 & 45765221584457950 & -10410420724200 \\
    14 & 226533008318367558727225 & 2456978448082146184175 &
      8727482292438634825 \\
    15 & 1212262421921910091110786350 & 30987133343627355178162050 &
      402136896268315968713050 \\
\hline
\hline
d & g=12 & g=13  & g=14 \cr
\hline
  1 & 0 & 0 & 0 \\
  : & : & : & : \\
    11 & 0 & 0 & 0 \\
    12 & 93900 & 750 & 0 \\
    13 & 140631287850 & -156011000 & -87500 \\
    14 & 4557518136088150 & 15789026415450 & -75723835350 \\
    15 & 2092760001092944591400 & 2911773889110540400 & 2180928121176200 \\
    16 & 85091316259943255956875200 & 617071664475129547125225 &
      1599147427910004974425 \\
\hline
 \end{array}
\]
}
\centerline{Table 3. BPS numbers $n_g^Y(d)$ of the covering $Y$ 
up to $g=14$.}

\newpage 

\section{ BPS numbers of $\tilde X_0$}

Gromov-Witten invariants $N_g^\bullet(\beta)$ of a Calabi-Yau manifold 
$\bullet$ are defined for $\beta \in H_2(\bullet, \mZ)$ in general. 
Corresponding BPS numbers are read by generalizing the relation 
(\ref{eqn:GW-BPS}). 
In the tables below, BPS numbers $n^{\tilde X_0}_{g}(i,j)$ are listed, where 
$(i,j)=(\beta .H_1, \beta .H_2)$ with the generators $H_1,H_2$ 
of $H^2(\tilde X_0,\mZ)$ from the each factor of $\mP^4 \times \mP^4$.

{\tiny
\[
  \begin{array}{|r|rcccccc|}
\hline
i \setminus j & 0  & 1  & 2 & 3 & 4 & 5 & 6  \\
\hline
0& 0 & 50 & 0 & 0 & 0 & 0 & 0 \\
1& 50 & 650 & 1475 & 650 & 50 & 0 & 0 \\
2& 0 & 1475 & 29350 & 148525 & 250550 & 148525 & 29350 \\
3& 0 & 650 & 148525 & 3270050 & 24162125 & 75885200 & 110273275 \\
4& 0 & 50 & 250550 & 24162125 & 545403950 & 5048036025 & 22945154050 \\
5& 0 & 0 & 148525 & 75885200 & 5048036025 & 114678709000 & 1231494256550 \\
6& 0 & 0& 29350 & 110273275 & 22945154050 & 1231494256550 & 27995704239850 \\
7& 0 & 0 & 1475 & 75885200 & 55531376500 & 7175800860250 & 334030085380350 \\
8& 0 & 0 & 0 & 24162125 & 74278763500 & 24352783493100 & 2329042266808650 \\
9& 0 & 0 & 0 & 3270050 & 55531376500 & 50034381769600 & 10084936612321850 \\
10& 0 & 0 & 0 & 148525 & 22945154050 & 63477362571125 & 28126794522576400 \\
11& 0 & 0 & 0 & 650 & 5048036025 & 50034381769600 & 51642034298930775 \\
12& 0 & 0 & 0 & 0 & 545403950 & 24352783493100 & 63157566038079800\\
\hline
   \end{array}
\]
}
\centerline{Table 4. BPS numbers $n_{g}^{\tilde X_0}(i,j)$ for $g=0$. }

{\tiny
\[
  \begin{array}{|r|cccccccc|}
\hline
i \setminus j & 0  & 1  & 2 & 3 & 4 & 5 & 6 & 7  \\
\hline
0&    0 & 0 & 0 & 0 & 0 & 0 & 0 & 0 \\
1&    0 & 0 & 0 & 0 & 0 & 0 & 0 & 0 \\
2&    0 & 0 & 0 & 0 & 0 & 0 & 0 & 0 \\
3&   0 & 0 & 0 & 1475 & 29350 & 148525 & 250550 & 148525 \\
4&    0 & 0 & 0 & 29350 & 2669500 & 46911250 & 303610050 & 882636150 \\
5& 0 & 0 & 0 & 148525 & 46911250 & 2311178040 & 38756326500 & 298784327925 \\
6&    0 & 0 & 0 & 250550 & 303610050 & 38756326500 & 1477879258975 &
      24724246516200 \\
7&    0 & 0 & 0 & 148525 & 882636150 & 298784327925 & 24724246516200 &
      824125289385950 \\
8&    0 & 0 & 0 & 29350 & 1249719025 & 1207298050100 & 217335663077200 &
      13948250904141600 \\
9&    0 & 0 & 0 & 1475 & 882636150 & 2731112702750 & 1103600201154950 &
      135767820281303350 \\
10&    0 & 0 & 0 & 0 & 303610050 & 3573290410020 & 3417167213249325 &
      817523002761866550 \\
11&    0 & 0 & 0 & 0 & 46911250 & 2731112702750 & 6658383337394000 &
      3186381984770132650 \\
12&    0 & 0 & 0 & 0 & 2669500 & 1207298050100 & 8301844531611000 &
      8273823575633968400 \\
\hline
   \end{array}
\]
}
\centerline{Table 5. BPS numbers $n_{g}^{\tilde X_0}(i,j)$ for $g=1$, }

{\tiny
\[
  \begin{array}{|r|cccccccc|}
\hline
i \setminus j & 0  & 1  & 2 & 3 & 4 & 5 & 6 & 7  \\
\hline
0&    0 & 0 & 0 & \;\;\;\;0 \;\;\;\;& 0 & 0 & 0 & 0 \\
1&    0 & 0 & 0 & 0 & 0 & 0 & 0 & 0 \\
2&    0 & 0 & 0 & 0 & 0 & 0 & 0 & 0 \\
3&    0 & 0 & 0 & 0 & 0 & 0 & 0 & 0 \\
4&    0 & 0 & 0 & 0 & 0 & 2500 & 36800 & 132150 \\
5&    0 & 0 & 0 & 0 & 2500 & 2238300 & 101188225 & 1276419800 \\
6&    0 & 0 & 0 & 0 & 36800 & 101188225 & 10885677450 & 321011805475 \\
7&    0 & 0 & 0 & 0 & 132150 & 1276419800 & 321011805475 & 19785981206800 \\
8&   0 & 0 & 0 & 0 & \;\;\; 191850 \;\;\;
& 6764994000 & 4007017841650 & 497495418446900 \\
9&   0 & 0 & 0 & 0 & 132150 & 17585700200 & 25512018106050 & 6354579702758450
      \\
10&  0 & 0 & 0 & 0 & 36800 & 24017901850 & 91091698085900 & 46325408551373425
      \\
11& 0 & 0 & 0 & 0 & 2500 & 17585700200 & 192086807308450 & 206305163005291900
      \\
12&  0 & 0 & 0 & 0 & 0 & 6764994000 & \;\; 245649059538250 \;\; & 
\; 585085137096464550 \;\\
\hline
   \end{array}
\]
}
\centerline{Table 6. BPS numbers $n_{g}^{\tilde X_0}(i,j)$ for $g=2$, }

\vskip0.5cm

\end{appendix}

\vskip1cm

\vskip1cm

\end{document}